\documentclass[a4paper,12pt]{amsart}
\usepackage{amsmath,amssymb,amscd,amsfonts}
\newtheorem{theorem}{Theorem}[section]

\newtheorem{rem}[theorem]{Remark}
\newtheorem{defn}[theorem]{Definition}
\newtheorem{lemma}[theorem]{Lemma}

\newtheorem{prop}[theorem]{Proposition}

\newtheorem{corollary}[theorem]{Corollary}

\def\ddbar{\partial\bar\partial}
\def\d{\partial}

\def\cG{{\mathcal G}}

\def\cC{{\mathcal C}}

\let\ol=\overline

\let\ep=\varepsilon

\def\a{\alpha}
\def\z{\zeta}

\def\g{\gamma}

\def\t{\delta}
\def\b{\beta}
\def\a{\alpha}

\title[Remark]
{Subharmonicity of conic Mabuchi's functional, I}

\author{Long Li}
  \address{Department of Mathematics and Statistics, McMaster University,
1280 Main Street West, Hamilton, ON L8S 4K1, Canada}
  \email[Long Li]{lilong@math.mcmaster.ca}

\begin{document}
\maketitle

\begin{abstract}
The purpose of this paper is to generalize the convexity of Mabuchi's functional to the conic setting. 
We first established a frame to study conic cscK metrics, and then the conic Mabuchi functional was introduced in such a way that 
conic cscK metrics are its critical points. Finally we proved the convexity of the 
conic Mabuchi functional along any $\mathcal{C}^{1,\bar{1}}_{\beta}$ geodesic. 
\end{abstract}

\section{Introduction}

Let $\mathcal{H}$ be the space of all K\"ahler metrics in a fixed cohomology class. In 1986,
Mabuchi \cite{M} introduced the so called ``$K$-energy map"(Mabuchi's functional) on this space, 
whose Euler-Lagrange equations are constant scalar curvature K\"ahler (cscK) metrics. In the same time, he observed that the complex Hessian of the $K$ energy map is semi-positive on $\mathcal{H}$. In 1987, he introduced an $L^2$ norm (Mabuchi's $L^2$ norm) \cite{M2} on the tagent space of $\mathcal{H}$, such that
$\mathcal{H}$ becomes an infinite dimensional Riemannian manifold. Then many people studied the geometry of this infinite dimensional Riemannian manifold 
including S. Semmes, S.K. Donaldson, and so on.

In 2000, X.X. Chen \cite{Chen00} proved that there exists a $\mathcal{C}^{1,\bar{1}}$ geodesic connecting any two points in $\mathcal{H}$. 
Mabuchi's observation combined with this result becomes very important in the study of uniqueness of canonical 
metrics. In fact, it was conjectured by X.X. Chen that Mabuchi's functional is a convex and continuous function along geodesics. 
However, the lack of regularities on the geodesic obstructed many attempts.

The breakthrough is a recent work by Berman and Berndtsson \cite{BB}. 
They first proved that Mabuchi's functional is weakly convex(without continuity) along any $\mathcal{C}^{1,\bar{1}}$ geodesic. 
Their beautiful idea is based on the Bergman kernel approximation 
and the subharmonicity of $log$-Bergman kernels first due to Yamaguchi.
Later they improved this work and proved the continuity.

In a joint work with X.X. Chen and Mihai P\u aun \cite{CLP}, we gave an independent proof of this conjecture from a different point of view. Our proof is based on a new approximation
of geodesics, inspired by Mihai's previous work on positivity of relative canonical bundles \cite{Mihai}.

The aim of this paper is to study this convexity property of Mabuchi's functional under conic setting. 
However,
the set up of conic cscK metrics is not completely clear before. 
For example, we can consider the following equation as a candidate:
\begin{equation}
\label{cke0}
Ric( \omega_{\beta})\wedge\omega_{\beta}^{n-1} = \omega_{\beta}^n + (1-\beta) [D]\wedge\omega_{\beta}^{n-1}.
\end{equation} 
Unfortunately, the wedge product $[D]\wedge \omega_{\beta}^n$ does not make sense even in Bedford-Talyor's setting. 
So the first step is to find a proper notion of conic cscK metrics. 
We discussed three different new notions(\emph{weak conic cscK, conic cscK, strong conic cscK metrics}) in Section \ref{sec-2}, 
and each of them has some independent interests. 

Based on this new definition of conic cscK metrics, Section \ref{sec-3} is devoted to study the corresponding conic energies.
The Euler-Lagrange equation of conic Mabuchi's functional is expected to be the conic cscK metrics. However, its integral form is first given in this
section, and we proved that conic cscK metrics are indeed critical points of conic Mabuchi's functional.

Our main result is proved in Section \ref{sec-4}. 
And the precise statement is as follows. 

\begin{theorem}
\label{thm-main}[theorem (\ref{555})]
Suppose there exists a conic $\mathcal{C}^{1,\bar{1}}_{\beta}$ geodesic $\mathcal{G}$ connecting two conic K\"ahler potential $\varphi_0$ and $\varphi_1$. 
Then the conic Mabuchi's functional $\mathcal{M}_c(t)$ restricted on $\mathcal{G}$ is a convex continuous function on the closed interval $[0,1]$.
\end{theorem}

The idea of the proof is very similar to our previous work \cite{CLP}, but difficulties arise in the ``a prior estimates"
for a family of Monge-Amp\`ere equations. Especially, there is no control on the lower bound of the curvature when the angle $\beta > 1/2$.
Guenancia and P\u aun developed a new technique \cite{GP14} on $C^2$ estimate of conic K\"ahler-Einstein metrics for full angles.
However, we can not incorporate their work directly into our picture,
since there are three unrelated smooth approximations (indexed as $\ep$, $\delta$ and $\ep'$)
 for the homogeneous complex Monge-Amp\`ere equation with cone singularity. 
The new trick here is to balance the growth rate of these three approximations to make them always comparable.
And this is done by a careful investigation of the regularization of conic geodesics.

Combined with the existence result of conic $\mathcal{C}_{\beta}^{1,\bar{1}}$ geodesic \cite{CZ12}, we immediately proved that
the conic Mabuchi's functional is a convex function between any two points in the space of conic K\"ahler potentials.
However, here we need to assume the angle $\beta$ is smaller than $1/2$. 
When the angle $\beta\in [1/2, 1)$, the existence of such geodesics is not clear up to this stage.

Finally we want to point out that we only consider simple smooth divisor $D$ as the set where conic singularity develops in this paper.
And also we assumed the line bundle $L_D$ induced from this divisor is semi-positive in the sense that there exists some smooth $psh$
metric on it. However, we believe that our method can be generalized to the smooth $klt$ setting, 
and the smoothness of the $psh$ metric on the line bundle can be removed. These results will come from sequel papers. 
\\
\\
$\mathbf{Acknowledgement}$: The author is very grateful to Prof. Xiuxiong Chen, for his generous help and continuous encouragement.
And the author would like to show his great thanks to Prof. Mihai P\u aun, for lots of useful discussion and valuable suggestions. 
The author also wants to thank Prof. Ian Hambleton and Prof. McKenzie Wang for their supports from many aspects. 

Part of this paper was completed during the author's visit to SNU, from a kind invitation by Prof. Dano Kim. 
And the author also wants to thank Prof. S. Kolodziej and Dr. N.C. Nguyen for useful discussion on pluripotential theory. 
Finally, the author wants to thank Prof. Song Sun, Prof. Kai Zheng and Dr. Chengjian Yao for their kind help 
on reading the draft of this paper and giving advices.

\section{Conic canonical metrics}
\label{sec-2}
Let $X$ be a compact complex $n$-dimensional K\"ahler manifold with K\"ahler form $\omega$. 
Suppose $D$ is a simple smooth divisor on the manifold $X$, and we assume the associated line bundle $L_D\geq 0$ is semi-positive,
in the sense that there exist a smooth $psh$ metric $\psi$ on it. 
And we also assume $D$ is induced by some holomorphic section $s\in H^0(X, L_D)$. 

When the first Chern class $c_1(X)$ of the manifold $X$ has a definite sign(positive, zero or negative), the most interesting canonical metrics are K\"ahler-Einstein metrics. We can define a K\"ahler-Einstein metric $\omega_{\varphi}$ by the following equation:
\begin{equation}
\label{ke}
Ric(\omega_{\varphi}) = \lambda \omega_{\varphi},
\end{equation}
where the constant $\lambda$ equals to $1(0, or -1)$, corresponding to the case $c_1(X)>0(=0, or <0)$.
However, when we investigate more general K\"ahler manifolds, this is just a special case of the so called $cscK\ metric$ (constant scalar curvature K\"ahler metric). 
A K\"ahler metric $\omega_{\varphi}$ is a $cscK\ metric$ if it satisfies the following equation on $X$:
\begin{equation}
\label{csck}
R_{\varphi} = n\underline{R},
\end{equation}
where the constant $\underline{R}$ is a topological invariant by Stokes' theorem, and it can be calculated as follows:
$$ \underline{R} = \frac{ [c_1(X)]\cdot [\omega]^{n-1}}{[\omega]^{n}}. $$ 
Notice that this is a 4th. order, fully non-linear partial differential equation, and it is equivalent to the following family of equations:
\begin{equation}
\label{csck1}
\left\{ \begin{array}{rcl}
Ric(\omega_{\varphi}) &=& \underline{R}\omega_{\varphi} + \theta \\
\theta\wedge\omega_{\varphi}^{n-1} &=& 0,
\end{array}\right.
\end{equation}
for some smooth closed $(1,1)$ form $\theta$,
on where the metric $\omega_{\varphi}$ is non-degenerate. After wedging $\omega_{\varphi}^{n-1}$ on both sides of the first equation in (\ref{csck1}), we get another equivalent equation of $cscK\ metrics$ as follows:
\begin{equation}
\label{csck2}
Ric(\omega_{\varphi})\wedge \omega_{\varphi}^{n-1} = \underline{R}\omega_{\varphi}^n.
\end{equation}

Equation (\ref{csck}), (\ref{csck1}) and (\ref{csck2}) indeed give three different notions of cscK metrics. 
When the metric is smooth, they are all equivalent. However, under the setting of conic K\"ahler metrics, the situation becomes more subtle as we will see later.

\subsection{Conic metrics in H\"{o}lder spaces} 
Let's consider all K\"ahler metrics on $X$ in a given cohomology class. This is equivalent to the space of all K\"ahler potentials as 
$$\mathcal{H}: = \{\varphi\in\mathcal{C}^{\infty}(X)|\ \omega_{\varphi}: = \omega+dd^c\varphi>0 \}. $$
Now we want to put some singularity to these metrics along the divisor $D$. The simplest case is when $X$ is the Euclidean space $\mathbb{C}^n$, and the divisor is defined by $\{ z_1 = 0 \}$. Then we define a model conic K\"ahler metric with cone angle $\beta$ along $D$ as
$$\omega_{\beta}: = \frac{1}{|z_1|^{2-2\beta}} dz_1\wedge d\bar{z}_1 + \sum_{j=2}^n dz_{j} \wedge d\bar{z}_j.$$
Back to general K\"ahler manifolds, we can also consider this kind of singular metrics by comparing it with the model metric locally near a point $p$ on the divisor.
\begin{defn}
\label{def-conic}
a conic K\"ahler metric $\omega_{\varphi}$ on $X$ with cone angle $\beta$ along the divisor $D$ is a K\"ahler current on $X$, smooth outside of $D$, such that
it is quasi-isometric to the model metric $\omega_{\beta}$ in a neighborhood of any point $p\in D$. 
\end{defn}
Take a point $p\in D$, there is an open coordinate chart $U\ni (z_1, z_2, \cdots, z_n)$ centered at $p$, such that the divisor is defined locally as $D:=\{z_1=0 \}$. Then
the quasi-isometric condition in definition (\ref{def-conic}) can be read as $$C^{-1}\omega_{\beta} \leq \omega_{\varphi} \leq C\omega_{\beta}, $$
for some constant $C$ on $U$.

According to Donaldson \cite{Don12}, we can introduce some H\"older norms to this space of conic K\"ahler metrics. 
Let's consider in the local coordinate chart $U\ni p$ first. Writing in polar coordinate $z_1 = \rho e^{i\theta}$, we can put
$ w = r e^{i\theta}$, where $r = \rho^{\beta}$. Hence $w_1 = w$ and $w_j = z_j$ for $j>1$ give another coordinate around $p$, and we call it as $w$-coordinate
(this change of coordinates is not holomorphic!). 
Then the natural cone metric in the new coordinate can be found as 
\begin{eqnarray}
g &=& dr^2 + \beta^2r^2d\theta^2 +  \sum_{j>1} (ds_j^2 + dt_j^2)
\nonumber\\
& = &\beta^2 \rho^{2\beta-2} d\rho^2 + \beta^2 \rho^{2\beta} d\theta^2 + \sum_{j>1} (ds_j^2 + dt_j^2)
\end{eqnarray}
where $z_j = s_j + \sqrt{-1}t_j$. Although it appears to be singular in the original coordinate, this cone metric preserves the finite distance near $p$.
Moreover, the metric $g$ is exactly the induced Riemannian metric by the skew-symmetric bilinear form $\omega_{\beta}$.
Hence we introduce a $1$ form as 
\begin{eqnarray}
\epsilon: &=&  e^{i\theta}(dr + \sqrt{-1} \beta r d\theta )
=  \beta \rho^{\beta-1} dz_1 
\nonumber\\
&=& \frac{1+\beta}{2} dw + \frac{1-\beta}{2}e^{2i\theta} d\bar{w},
\end{eqnarray}
and $$ \epsilon\wedge\ol{\epsilon} = \beta^2 \rho^{2\beta -2} dz_1\wedge d\bar{z}_1 = \beta dw\wedge d\bar{w}.$$
Notice that in $w$-coordinate, the cone metric $g$ is quasi-isometric to the Euclidean metric. 
Then according to Donaldson, for $\alpha < \beta^{-1} -1$, we make the following definitions.
\begin{defn}
\label{def-1}
a real-valued function $f$ is in $\mathcal{C}^{0,\alpha,\beta}$, if $f$ is a $C^{\alpha}$ H\"older continuous function in $w$-coordinate.
\end{defn}

A $(1,0)$ form $\xi$ on $X$(in the original coordinate chart) locally can be written as 
$$ \xi = f \epsilon + \sum_{j>1} g_j dz_j. $$ 
Then we can introduce another H\"older space for $(1,0)$ forms as 
\begin{defn}
\label{def-2}
a $(1,0)$ form $\xi$ is in $\mathcal{C}^{,\alpha,\beta}$ if $f,g\in \mathcal{C}^{0,\alpha,\beta}$ and the limit of $f$ vanishes when it approaches the divisor. 
And a function $f\in\mathcal{C}^{1,\alpha,\beta} $ means $f\in \mathcal{C}^{0,\alpha,\beta}$ and $\partial f\in \mathcal{C}^{,\alpha,\beta}$. 
\end{defn}
It is equivalent to the following conditions in $w$-coordinate:
$$ \mathcal{C}^{1,\alpha,\beta} = \{ f\in\mathcal{C}^{1,\alpha}, \ \ \ \frac{ \partial f}{\partial w} (0,w') =0,\  \frac{ \partial f}{\partial \bar{w}} (0,w') =0\}.$$
This is a Banach space inherited from the usual $\mathcal{C}^{1,\alpha}$ norm in $w$-coordinates. 
Now for a $(1,1)$ form $\tau$ on $U$, we have 
$$\tau = f \epsilon\wedge \ol{\epsilon} + \sum_{j>1} (g_{1j}\epsilon\wedge d\bar{z}_j + g_{j1} dz_j\wedge \ol{\epsilon}) + \sum_{j,k>1}g_{jk}dz_j\wedge d\bar{z}_k.$$

Then let's define a new H\"older space as follows.
\begin{defn}
\label{def-3}
a $(1,1)$ form $\tau$ is in $\mathcal{C}^{,\alpha,\beta}$, if $f, g_{jk}\in\mathcal{C}^{0,\alpha,\beta}$ for all $j, k$, and $g_{1j}, g_{j1}$ converges to zero when they approach the divisor. And a function $f\in \mathcal{C}^{2,\alpha,\beta}$ if $f\in\mathcal{C}^{0,\alpha,\beta}$, $\partial f\in \mathcal{C}^{,\alpha,\beta}$ 
and $\ddbar f\in\mathcal{C}^{,\alpha,\beta}$. 
\end{defn}


Again this is a Banach space by restricted $\mathcal{C}^{2,\alpha}$ norm from $w$-coordinates.
Finally, we can further define a higher regularity class \cite{CZ12} as follows 
$$\mathcal{C}^{3,\beta}: = \{f\in\mathcal{C}^{2,\alpha,\beta} | \ \ \text{all the third derivatives of $f$ is bounded w.r.t $\omega_{\beta}$}.  \} $$

According to Calami and Zheng \cite{CZ12}, this higher regularity space $\mathcal{C}^{3,\beta}$ is useful, when we consider the existence of geodesics connecting
points in the space of conic K\"ahler metrics. And we will give further details when this result is mentioned later.

\subsection{Conic cscK metrics}
Our goal is to give a proper notion of cscK metric under the conic setting. The cheapest way is to forget about the singular part of the manifold, and define it as the usual cscK metric on the regular part. 
\begin{defn}
\label{def-ccscK0}
a weak conic cscK metric is a conic K\"ahler metric  $\omega_{\varphi}$ with cone angle $\b$ along the divisor $D$, satisfying the following equation on $X-D$:
\begin{equation}
\label{ccscK0}
R_{\varphi} = c,
\end{equation}
where $c$ is a proper constant. 
\end{defn}
This constant $c$ should only depend on the K\"ahler class$( [\omega],  [\omega\vert_{D}] )$, the first Chern class$(c_1(X), c_1(D))$, and the class of the divisor $D$.
However, this is difficult to know, even if $\varphi\in\cC^{2,\a,\b}$. In general, Stokes' theorem(integration by parts) fails to be applied to the LHS of equation of (\ref{csck2}) without further regularity assumptions.

In order to overcome this difficulty, we are going to invoke some pluripotential theory from the work of Boucksom-Eyssidieux-Guedj-Zeriahi (\cite{BEGZ10}). 
Recall that two $quasi$-$psh$ functions $\varphi$ and $\psi$ have the same singular type 
if $\varphi = \psi  + O(1)$. And the singular type of a positive closed $(1,1)$ current is determined by its potential. 

\begin{lemma}
\label{lem-ric}
Let $\omega_{\varphi}$ be a conic K\"ahler metric, and $\Omega$ be another conic K\"ahler metrics with $\mathcal{C}^{2,\alpha,\beta}$ potentials. 
Assume that on $X-D$, we have 
$$ Ric(\omega_{\varphi})\geq -C\Omega, $$ 
for some large constant $C$, such that $Ric(\omega) + C\omega >0$.
Then we can extend the $(1,1)$ form 
$$Ric(\omega_{\varphi}) + C\Omega$$
as a positive closed current on $X$.
Moreover, the function $\phi_{ric}: = -\log(\omega_{\varphi}^n / \omega^n)$ determines the singularity type of this positive current. 
\end{lemma}
\begin{proof}
Write $\Omega = \omega + dd^c\psi$, for some $\psi\in \mathcal{C}^{2,\alpha,\beta}$. We can re-write the positivity condition $Ric(\omega_{\varphi}) + C\Omega \geq 0$ as
\begin{equation}\label{ric}
\Gamma + dd^c( -\log\frac{\omega_{\varphi}^n}{\omega^n} + C\psi) \geq 0, 
\end{equation}
where $\Gamma = Ric(\omega) + C\omega > 0$. Notice that the potential $\psi$ is continuous on $X$, and we claim that the function $\phi_{ric}$ is bounded from above near the divisor. Hence the new defined function $\Psi: = \psi + \phi_{ric}$ can be extended as a $\Gamma$-$psh$ function across the divisor.
Writing $\phi_{ric} = \Psi - \psi$, we can extend $\phi_{ric}$ across the divisor as the difference of two $quasi$-$psh$ functions, and it is uniformly bounded from above since 
$\psi$ is a bounded function on $X$. Finally the singularity type of $\Psi$ is determined by $\phi_{ric}$ for the same reason.

Now we are going to prove the claim. It's enough to consider the upper bound of $\phi_{ric}$ in a local coordinate chart of a point $p$ on the divisor $D$. Then by the local isometric property of conic K\"ahler metric, we can compare $\omega_{\varphi}$ to the model metric $\omega_{\beta}$ in this neighborhood. That is, there exist a constant $C$, such that 
$$C^{-1} \omega_{\beta}^n \leq \omega_{\varphi}^n \leq C\omega_{\beta}^n. $$
But in the local coordinate chart, we have 
$$-\log\frac{\omega_{\beta}^n}{\omega^n} \sim (1-\beta)\log|z_1|^2 + O(1),$$
which is obvious bounded from above. Therefore, the claim follows. 

\end{proof}

From now on, we can consider the space of all conic K\"ahler metrics with Ricci curvature bounded from below. 
\begin{defn}\label{def-ric lower}
a conic K\"ahler metric $\omega_{\varphi}$ is in the family $\mathcal{R}$ if there exists some(may depend on $\omega_{\varphi}$) 
conic K\"ahler metric $\Omega$ with $\mathcal{C}^{2,\alpha,\beta}$ potential such that 
\begin{equation}
\label{ric lower bound }
Ric(\omega_{\varphi})\geq -C\Omega
\end{equation}
for some large constant $C$. 
\end{defn}

This extra positivity imposing on the Ricci curvature is equivalent to say that the Ricci potential $\phi_{ric}$ is a $quasi$-$psh$ function on $X$.
Moreover, observe that the Ricci potential $\phi_{ric}$ and the K\"ahler potential $\varphi$ for a $\cC^{2,\a,\b}$ conic metric $\omega_{\varphi}$ 
both have small unbounded locus \cite{BEGZ10}, since the smooth divisor $D$ is a complete pluripolar closed subset of $X$.  
Then we can invoke the integration by parts formula (Theorem 1.14, \cite{BEGZ10}) to compute the topological constant. 

\begin{prop}
\label{IBP0}
Let $\omega_{\varphi_1}, \omega_{\varphi_2}\in\mathcal{R}$ are two conic K\"ahler metrics with $\mathcal{C}^{2,\alpha,\beta}$ potentials. 
Then we have the following equation:
\begin{equation}
\label{IBP}
\int_{X-D} Ric(\omega_{\varphi_1})\wedge \omega_{\varphi_1}^{n-1} = \int_{X-D} Ric(\omega_{\varphi_2})\wedge \omega_{\varphi_2}^{n-1}
\end{equation}
\end{prop}
\begin{proof}
Notice that for two conic K\"ahler metrics $\omega_{\varphi_1}$ and $\omega_{\varphi_2}$, the following is always true: 
there exists some constant $K>0$, such that the difference of their $Ricci$ potentials is globally bounded by $K$, i.e. 
\begin{equation}
\label{ricbdd}
|\phi_{1,ric} - \phi_{2,ric}| = \left| \log\left(\frac{\omega_{\varphi_1}^n}{\omega_{\varphi_2}^n}\right) \right| < K,
\end{equation}
on $X$. This is true because these two Ricci potentials have the same singularity type near the divisor according to lemma (\ref{lem-ric}). 
In fact, in a local coordinate chart $U$ near a point $p\in D$, we can compare the two conic metrics with the model metric $\omega_{\beta}$ in this neighborhood. 
That is, there exist a constant $C$, such that 
$$C^{-1} \omega_{\beta}^n \leq \omega_{\varphi_i}^n \leq C\omega_{\beta}^n,  $$
where $i=1,2$. Writing 
$$u = \phi_{ric, 1} - \phi_{ric, 2} = ( \log \omega_{\varphi_1}^n - \log \omega_{\beta}^n) -( \log \omega_{\varphi_2}^n - \log \omega_{\beta}^n), $$
we can see $|u|\leq 2C$ on $U$, and hence bounded by some constant $K$ on $X$. 

Then we will illustrate the idea in the case when $n=2$, and the general case follows similarly.
Let's first claim that the following integration by parts equation holds:
\begin{equation}\label{IBP1}
\int_{X-D} (Ric(\omega_{\varphi_1}) - Ric(\omega_{\varphi_2}))\wedge\omega_{\varphi_i} = 0,
\end{equation}
where $i=1, 2$. Let's define two new positive currents as $T_i = Ric(\omega_{\varphi_i}) + C\Omega$. Here we can always assume the constant $C$ is big enough such that $Ric(\omega) + C\omega >0$. 
Writing $T_1 = \theta + dd^c\psi_1$ and $T_2= \theta + dd^c\psi_2$, where $\theta$ is a smooth closed $(1,1)$ form in the class $-c_1(X) + C[\Omega]$. 
Then equation (\ref{IBP}) is reduced to show 
$$ \int_{X-D} (dd^c\psi_1 - dd^c\psi_2)\wedge\omega_{\varphi_i} = 0. $$
Notice that $\psi_1$ and $\psi_2$ are two $quasi$-$psh$ functions on $X$, which are smooth on $X-D$. And their difference 
$u = \psi_1 - \psi_2 $ is uniformly bounded on $X$, because their currents have the same singularity type thanks to lemma (\ref{lem-ric}) and 
equation (\ref{ricbdd}). Then equation (\ref{IBP1}) follows from Theorem 1.14 of \cite{BEGZ10}. 

Next the following equation follows from a similar argument:
\begin{equation}
\label{IBP2}
\int_{X-D} Ric(\omega_{\varphi_i})\wedge ( \omega_{\varphi_1} - \omega_{\varphi_2} ) = 0,
\end{equation}
for $i=1,2$. Combine equations (\ref{IBP1}) and (\ref{IBP2}) together, and we proved equation (\ref{IBP}) when $n=2$.
\end{proof}

Based on this integration by parts, we can introduce the following definition.

\begin{defn}
a conic cscK metric with cone angle $\beta$ along the divisor $D$ is a conic K\"ahler metric $\omega_{\varphi}$ with $\mathcal{C}^{2,\alpha,\beta}$ potential, 
such that $\omega_{\varphi}\in \mathcal{R}$, and 
$$ R_{\varphi} = c(\beta), $$ on $X-D$.
\end{defn}

\begin{corollary}
\label{cor-ccsck1}
The constant $c(\beta)$ in the definition of conic cscK metrics is a topological invariant. 
\end{corollary}
\begin{proof}
Suppose $\omega_{\varphi_1}$ and $\omega_{\varphi_2}$ are two conic cscK metrics with cone angle $\beta$ along $D$. 
They are both in the family $\mathcal{R}$. 
Then it's enough to prove equation (\ref{IBP}) since the constant can be written as
$$c_i(\beta) = \frac{\int_{X-D} Ric(\omega_{\varphi})\wedge\omega_{\varphi}^{n-1}}{\int_{X-D}\omega_{\varphi}^n},$$ for $i=1,2$. 
But this is true thanks to Proposition (\ref{IBP0}). 
\end{proof}

This definition is a kind of generalization of equation (\ref{csck}) from smooth cscK metrics, but we can also generalize it from equation (\ref{csck1}). 

\begin{defn}
\label{def-ccsck2}
a strong conic cscK metric with cone angle $\beta$ along the divisor $D$ is a conic K\"ahler metric $\omega_{\varphi}$ with $\mathcal{C}^{2,\alpha,\beta}$ potential, such that it satisfies the following family of equations in the sense of currents
\begin{equation}
\label{ccsck2}
\left\{ \begin{array}{rcl}
Ric(\omega_{\varphi}) &=& c(\beta) \omega_{\varphi} + \Theta + (1-\beta)[D] \\
\Theta\wedge\omega_{\varphi}^{n-1} &=& 0,
\end{array}\right.
\end{equation}
where $\Theta$ is a closed $(1,1)$ current, such that $$\Theta=\theta+ dd^c\psi, $$ where $\theta$ is a smooth closed $(1,1)$ form on $X$ in the same cohomology class, and $\psi\in \mathcal{C}^{2,\alpha,\beta}$. 
\end{defn}

\begin{corollary}
\label{cor-ccsck}
The constant $c(\beta)$ in the definition of strong conic cscK metrics is a topological invariant. 
\end{corollary}
\begin{proof}
It's enough to prove that $\omega_{\varphi_1}$ and $\omega_{\varphi_2}$ are both in the family $\mathcal{R}$.
This is because we can estimate the lower bound of their Ricci currents as 
$$ Ric(\omega_{\varphi_i}) \geq c_i(\beta)\omega_{\varphi_i} + \Theta_i \geq -C\Omega, $$
where $\Omega =(\omega_{\varphi_1} + \omega_{\varphi_2}) /2$ is a conic K\"ahler metric with $\mathcal{C}^{2,\alpha,\beta}$ potential, and $C$ is a large constant. 
Now thanks to lemma (\ref{lem-msr}) and lemma (\ref{lem-well}), we can form the non-pluripolar product on both sides of the first equation of (\ref{ccsck2}) as 
$$ Ric(\omega_{\varphi})\wedge\omega_{\varphi}^{n-1}  = c(\beta) \omega_{\varphi}^n + (1-\beta) [D]\wedge\omega_{\varphi}^{n-1}. $$
Then the result follows from Proposition (\ref{IBP0}) again.
\end{proof}

Notice that the first equation in (\ref{ccsck2}) is in fact a twisted conic K\"ahler-Einstein equation, so it is equivalent to the following Monge-Amp\`ere equation:
\begin{equation}
(\omega+dd^c\varphi)^n = \frac{e^{h - c(\beta)\varphi} \omega^n}{|s|^2e^{-\psi}},
\end{equation} 
with normalization condition $\int_X \frac{e^{h - c(\beta)\varphi}\omega^n}{|s|^2e^{-\psi}} = 1 $. Here $\psi$ is a smooth $psh$ metric for the line bundle $L_D$, and the function $h = h_{\omega,\Theta} \in \mathcal{C}^{2,\alpha,\beta}$ solves the following equation:
\begin{equation}
\label{h}
Ric(\omega) = c(\beta)\omega + dd^c\psi + \Theta + dd^c h_{\omega,\Theta}.
\end{equation}

\begin{rem}
\label{rem-1}
Strong conic cscK metrics are certainly conic cscK metrics. However, it is interesting to know if a conic cscK metric is strong conic cscK. 
\end{rem}

\subsection{distributional view}

Another attempt is motivated by the definition of conic K\"ahler-Einstein metrics. A conic K\"ahler-Einstein metric with cone angle $\beta$ along the divisor $D$ is 
a conic K\"ahler metric $\omega_{\varphi_{\beta}}$ with $\mathcal{C}^{2,\alpha,\beta}$ potential, 
satisfying the following equation on $X$ in the sense of currents:
\begin{equation}
\label{cke}
Ric( \omega_{\varphi_{\beta}}) = c(\beta)\omega_{\varphi_{\beta}} + (1-\beta) [D],
\end{equation} 
where $[D]$ is the integration current of the divisor $D$, and the constant $c(\beta)$ is determined by the correct cohomology condition. Now we can 
wedge a positive $(n-1, n-1)$ form $\omega_{\varphi_{\beta}}^{n-1}$ on both sides of equation (\ref{cke}):
\begin{equation}
\label{ccsck1}
Ric(\omega_{\varphi_{\beta}})\wedge\omega_{\varphi_{\beta}}^{n-1}  = c(\beta) \omega_{\varphi_{\beta}}^n + (1-\beta) [D]\wedge\omega_{\varphi_{\beta}}^{n-1}. 
\end{equation}

This equation is a natural generalization of conic K\"ahler-Einstein equation. However, it is not clear that the wedge product 
$$[D]\wedge\omega_{\varphi}^{n-1} $$ is well defined in prior. 

Recall that the non-pluripolar product for unbounded $psh$ functions is defined in \cite{BEGZ10}. Generally speaking, we can take the trivial extension of the usual wedge product across the pluripolar set, provided it always has local finite mass near the pluripolar set. 
\begin{lemma}
\label{lem-msr}
The non-pluripolar product $\langle [D]\wedge\omega_{\varphi}^{n-1} \rangle$ is well defined for any conic K\"ahler metric $\omega_{\varphi}$ with 
$\mathcal{C}^{2,\alpha,\beta}$ potential. In particular, the non-pluripolar product is a zero current on $X$.
\end{lemma}
\begin{proof}
The divisor $D$ is a (locally) complete pluripolar closed subset of $X$.
According to \cite{BEGZ10}, 
the potential of the current $[D]$ is a
$quasi$-$psh$ function with small unbounded locus. And $\varphi\in\mathcal{C}^{2,\alpha,\beta}$ is bounded on $X$.
Then it's enough to prove the Bedford-Talyor product $[D]\wedge\omega_{\varphi}^{n-1}$ defined on the open subset $X-D$ has locally finite mass near each point of the divisor. 

By $Poincar\acute{e}$-$Lelong$ formula, the integration current $[D]$ is equal to $dd^c\log |s|^2$, where $s$ is its defining section. 
In a local coordinate $(z_1,\cdots, z_n)$ near a point $p\in U$ on $D$, we can assume the divisor $D$ coincides with the set $\{z_1 =0\}$, and the integration current can be written down as $$[D] = dd^c\log |z_1|^2, $$ in this local coordinate. Define 
$v_{\ep} = \log(|z_1|^2 + \ep) $ in the open neighborhood $U$.
Then $\{v_{\ep}\}$ forms a sequence of $psh$ functions decreasing to $\log |z_1|^2$ locally. And notice that 
\begin{equation}
\label{ddv}
\ddbar v_{\ep} = \frac{ \ep dz_1\wedge d\bar{z}_1}{(|z_1|^2+\ep)^2} 
\end{equation}
is a positive current on $X-D$. Hence we can compute the integral as 
\begin{eqnarray}
\label{well define}
\int_{U-D} dd^c v_{\ep} \wedge \omega_{\varphi}^{n-1} &=& \int_{U-D} dd^c v_{\ep} \wedge \Pi (\omega_{\varphi}^{n-1})
\nonumber\\
&=&\int_{U-D} dd^c v_{\ep} \wedge (\Pi\omega_{\varphi})^{n-1},
\end{eqnarray}
where the operator $\Pi$ is to project any $(p,p)$ form to parallel direction of the divisor(delete all forms involving $dz_1$ or $d\bar{z}_1$). 
This is because the $(1,1)$ form $dd^c v_{\ep}$ already occupies all the normal directions($dz_1$ and $d\bar{z}_1$) of the hypersurface $D$ of $X$.  
Now notice that the $(n-1, n-1)$ current $(\Pi\omega_{\varphi})^{n-1}$ is uniformly bounded on $X$. Therefore, by Fubini's theorem, the integral in equation (\ref{well define}) can be estimated as
\begin{eqnarray}
\label{well define2}
&\left( 2\pi\ep\int_0^1 \frac{rdr}{(r^2+\ep)^2}\right)& \left(\int_{|z_2|^2 + \cdots + |z_n|^2 \leq 1} dz_2\wedge d\bar{z}_2\cdots dz_n\wedge d\bar{z}_n \right) 
\nonumber\\
&\leq&  \left. \frac{-\ep C}{2(r^2 + \ep)} \right|_0^1
\nonumber\\
&=& \frac{C}{2(1+\ep)},
\end{eqnarray}
where $|z_1| = r$ and $C$ is some uniform constant. Finally, our integral is bounded by Bedford-Talyor's approximation theorem:
$$ \int_{U-D} [D]\wedge \omega_{\varphi}^{n-1} =\lim_{\ep} \int_{U-D} dd^c v_{\ep}\wedge \omega_{\varphi}^{n-1} < +\infty. $$

Finally, the non-pluripolar product $\langle [D]\wedge\omega_{\varphi}^{n-1}\rangle$ charges zero mass on the divisor, since $D$ is a pluripolar set in $X$. On the other hand, its support is in the divisor. In fact, it's the trivial extension of the zero current $\left. [D]\wedge\omega_{\varphi}^n \right|_{X-D}$ to $X$.
\end{proof}

\begin{rem}
There is another way to define this wedge product by Demailly's generalization of Bedford-Talyor product, but this product is no longer trivial on the divisor. 
\end{rem}
Now the non-pluripolar product $\langle [D]\wedge\omega_{\varphi}^{n-1}\rangle$ is a closed positive $(n,n)$ current on $X$. That is, its coefficient is a positive measure on the manifold. Then equation (\ref{ccsck1}) can hold in the sense of measures, and this naturally implies the wedge product between the Ricci curvature and $\omega^{n-1}_{\varphi}$ is well defined. 

\begin{prop}
\label{lem-well}
Suppose the conic K\"ahler metric $\omega_{\varphi}$ is in the family $\mathcal{R}$. Then the non-pluripolar product $Ric(\omega_{\varphi})\wedge\omega_{\varphi}^{n-1}$ is well defined.
\end{prop}
\begin{proof}
Write $Ric(\omega_{\varphi}) = \Gamma + dd^c\Psi - C\Omega$, and the Ricci potential $\phi_{ric} = \Psi - \psi$ as in lemma (\ref{lem-ric}). 
Then it's enough to prove the non-pluripolar product $(\Gamma+ dd^c \Psi ) \wedge \omega_{\varphi}^{n-1}$ is well defined. In a local coordinate chart $U$ near a point 
$p$ on the divisor $D$, we can write $dd^c\tilde{\Psi} = \Gamma + dd^c \Psi$ for some $psh$ function $\tilde{\Psi}$ in $U$. 
Notice that $\tilde{\Psi}$ has the same singularity type with the local function $(1-\beta)\log |s|^2$. Hence the new defined function $\Phi = \tilde{\Psi} - (1-\beta)\log |s|^2$ is uniformly bounded on $X$. Now let $\chi\geq 0$ be a smooth function on $U$ with compact support, and we can estimate their local mass as
\begin{equation}
\label{finite mass}
\int_{U-D} \chi dd^c\tilde{\Psi}\wedge\omega_{\varphi}^{n-1} = \int_{U-D} \chi dd^c\Phi\wedge\omega_{\varphi}^{n-1} 
+ (1-\beta)\int_{U-D} \chi[D]\wedge\omega_{\varphi}^{n-1}.
\end{equation}

Notice that the first term of the RHS of equation (\ref{finite mass}) is bounded from integration by parts:
$$ \int_{U-D} \chi dd^c\Phi\wedge\omega_{\varphi}^{n-1} = \int_{U-D} \Phi dd^c\chi \wedge \omega_{\varphi}^{n-1}.$$
This is true thanks to theorem 1.14 of \cite{BEGZ10} again. And the second term of the RHS of equation (\ref{finite mass}) is also bounded 
because of lemma (\ref{lem-msr}). Then the result follows.
\end{proof}

\section{Conic energies}
\label{sec-3}

In the space of smooth K\"ahler potentials(metrics) $\mathcal{H}$, we can define the following energy functionals for any $\varphi\in \mathcal{H}$:
\begin{equation}\label{energy1}
\mathcal{E}(\varphi): =\frac{1}{n+1} \int_X \sum_{j=0}^n \varphi\omega_{\varphi}^{n-j}\wedge\omega^{j},
\end{equation}
and
\begin{equation}
\label{energy2}
\mathcal{E}^{\alpha}(\varphi): = \int_X \sum_{j=0}^{n-1}\varphi \omega_{\varphi}^{n-1-j}\wedge\omega^{j}\wedge\alpha,
\end{equation}
where $\alpha$ is any closed smooth $(1,1)$ form on $X$. Let $\varphi_t$ be a curve in the space $\mathcal{H}$, we can take the first variation of the two energies as 
\begin{equation}
\label{energy3}
\frac{d}{dt}\mathcal{E}(\varphi_t) = \int_X \varphi'_t\omega_{\varphi_t}^n;\ \ \ \ \frac{d}{dt}\mathcal{E}^{\alpha}(\varphi_t) = \int_X \varphi'_t\omega_{\varphi}^{n-1}\wedge\alpha.
\end{equation}
And on the product space $X\times \Sigma$, we can compute their complex Hessians as 
\begin{equation}
\label{hessian}
d_{\tau}d^c_{\tau} \mathcal{E}(\varphi_{\tau}) = \int_{X} (\pi^*\omega + dd^c\Phi)^{n+1}; \ \ \ 
d_{\tau}d^c_{\tau} \mathcal{E}^{\alpha}( \varphi_{\tau} )= \int_{X} (\pi^*\omega + dd^c\Phi)^{n}\wedge\pi^*\alpha.
\end{equation}
In general, the first energy $\mathcal{E}$ is an affine function along geodesics, and both energies $\mathcal{E}$ and $\mathcal{E}^{\alpha}$
are continuous with respect to weak topology.
Now, the so called Mabuchi's functional is defined as 
$$\mathcal{M}: = E + H, $$
where $E = \underline{R}\mathcal{E} - \mathcal{E}^{Ric\omega} $ is called the energy part of $\mathcal{M}$, and $H$ is the entropy part, which is defined as 
\begin{equation}
\label{entropy}
H(\varphi): = \int_X \log\frac{\omega_{\varphi}^n}{\omega^n} \omega_{\varphi}^n
\end{equation}
Moreover, the first variation of Mabuchi functional can be computed as 
\begin{equation}
\label{K}
\frac{d}{dt}\mathcal{M}(\varphi_t) = n\int_X \varphi'_t (\underline{R}\omega_{\varphi} - Ric\omega_{\varphi})\wedge\omega_{\varphi}^{n-1}.
\end{equation}

Our next goal is to generalize these energy functionals to conic setting.
But before doing that, let's discuss a bit about conic geodesics in the space of conic K\"ahler metrics.

\subsection{Conic geodesics}
Donaldson \cite{Don12} defines a global model metric with cone angle $\beta$ as
\begin{equation}
\label{533}
\tilde{\omega}_{\beta}: =  \omega + dd^c\varphi_{\beta}>0 ,
\end{equation}
where $ \varphi_{\beta} = C^{-1} dd^c |s|^{2\beta}_{\psi} $ with a large constant $C$. This global model metric is equivalent to conic K\"ahler metrics up to quasi-isometry.
\begin{defn}
\label{def-c11}
Let $\omega_{\varphi}$ be a possibly degenerate  K\"ahler current. 
Then we say its potential $\varphi$ has bounded conic Laplacian if 
$\Delta_{\tilde{\omega}_{\beta}}\varphi $ is an $L^{\infty}$ function on $X$. And we can write $\varphi \in \mathcal{H}_{\Delta,\beta}$.
\end{defn}

Thanks to the $L^p$ estimates and embedding theorems established for conic K\"ahler metrics \cite{CW},
the potential $\varphi \in \mathcal{H}_{\Delta, \beta}$ is a $C^{1,\alpha}$ function in $w$-coordinate.

Now given two conic K\"ahler metrics(potentials) $\varphi_0$ and $\varphi_1$, we can consider a curve ${\varphi_t}$ connecting them. Instead of the real variable $t$, it's better to introduce another
irrelevant real variable to make it as a complex variable $$\tau: = e^{t + is}. $$ Hence our domain under consideration becomes a product manifold $X\times\Sigma$, 
where $\Sigma$ is an annuals in $\mathbb{C}$.

On the product manifold $X\times\Sigma$, we call a metric $\mathcal{G}$ to be conic $\mathcal{C}^{1,\bar{1}}_{\beta}$ if 
$$0\leq \mathcal{G} \leq C \mathcal{G}_{\beta}, $$
where $C$ is a fixed constant, and $\mathcal{G}_{\beta}: = \pi^* \tilde{\omega}_{\beta} + \sqrt{-1} dt\wedge d\bar{t}$. 
Hence $\mathcal{G}|_{X\times\{t\}}\in\mathcal{H}_{\Delta, \beta}$ for each $t\in\Sigma$. 

\begin{defn}
\label{def-con}
A conic $\mathcal{C}_{\beta}^{1,\bar{1}}$ geodesic $\mathcal{G}$(under Mabuchi's $L^2$ norm metric) connecting $\varphi_0$ and $\varphi_1$
is a closed positive $(1,1)$ current on $X\times\Sigma$ with $\mathcal{C}_{\beta}^{1,\bar{1}}$ regularity, satisfying
the following conditions:
\begin{eqnarray}
\label{conic-geod}
\mathcal{G}^{n+1} &=& 0
\nonumber\\
\mathcal{G}|_{X\times\{ 0\}, \{ 1\}} &=& \omega_{\varphi_{0,1}}
\end{eqnarray}
\end{defn}

In order to establish conic geodesics, we need to introduce more regularities on the boundary of $X\times\Sigma$. 
According to \cite{CZ12}, we denote $\mathcal{H}_C$ by the space of all
K\"ahler potentials $\varphi\in\cC^{3,\b}$ such that $\omega\in\mathcal{R}$, for $\b\in (0,\frac{1}{2})$. 
Then we have the following existence theorem for conic geodesics. 

\begin{theorem}[Calami-Zheng]
\label{CZ}
Assume two K\"ahler potentials $\varphi_0, \varphi_1$ are in the space $\mathcal{H}_C$. 
Then there exists a conic $\mathcal{C}^{1,\bar{1}}_{\beta}$ geodesic $\mathcal{G}$ connecting them. 
\end{theorem}

When the angle $\beta \geq 1/2$, the unboundedness of curvature causes troubles in the proof of such existence theorem.
However, it is still conjectured that such conic geodesic exists for every angle $\beta\in (0,1)$, 
and for all points in the space of $\mathcal{C}^{2,\alpha,\beta}$ conic K\"ahler potentials. 

\subsection{Conic Mabuchi's functional}

Now we can try to generalize our energy functionals to the space of conic metrics with $\mathcal{C}^{2,\alpha,\beta}$ potentials. The goal is to find the Euler-Langrange equation for conic cscK metrics defined in (\ref{ccsck2}).
 
First notice that for $\mathcal{C}^{2,\alpha,\beta}$ potentials, the energies $\mathcal{E}$ and $\mathcal{E}^{\alpha}$ are well defined thanks to the Bedford-Taylor product. And we can compute their first variation formula as follows.
\begin{lemma}
\label{lem-1variation}
Suppose $\varphi_t$ is a $\mathcal{C}^1$ curve in the space of conic K\"ahler potentials with $\mathcal{C}^{2,\alpha,\beta}$ H\"older norm. Then equation (\ref{energy3}) holds.
\end{lemma}
\begin{proof}
First note that the Bedford-Talyor product $\omega_{\varphi}^{n-j}\wedge\omega^j$ doesn't charge mass on any pluripolar set, and $\varphi$ is a continuous function on $X$. Hence we can write the energy integration as 
$$\int_X \varphi\omega_{\varphi}^{n-j}\wedge\omega^j = \int_{X-D}\varphi \omega_{\varphi}^{n-j}\wedge\omega^j. $$

Now it's enough to prove the following claim:
\begin{equation}
\label{IBP-time}
\int_{X-D}\varphi dd^c\varphi' \wedge \omega_{\varphi}^{n-j-1}\wedge\omega^j  = \int_{X-D} \varphi' dd^c\varphi\wedge\omega_{\varphi}^{n-j-1}\wedge\omega^j.
\end{equation}
The time derivative $\varphi' \in\mathcal{C}^{2,\alpha,\beta}$ is in the tangent space of all $\mathcal{C}^{2,\alpha,\beta}$ conic K\"ahler potentials, and it satisfies the following convergence condition:
\begin{equation}
\label{C1}
\left| \frac{1}{t} (\varphi_t -\varphi -t\varphi')  \right|_{\mathcal{C}^{2,\alpha,\beta}} \rightarrow 0.
\end{equation}
Then by dominate convergence theorem, we can write the RHS of equation (\ref{IBP-time}) as 
\begin{equation}
\label{IBP-time0}
 \int_{X-D} \varphi' dd^c\varphi\wedge\omega_{\varphi}^{n-j-1}\wedge\omega^j=\lim_{t\rightarrow 0} \frac{1}{t}\int_{X-D} (\varphi_t - \varphi) dd^c\varphi\wedge\omega_{\varphi}^{n-j-1}\wedge\omega^j.
\end{equation}
Next, we can normalize our potential $\varphi \geq 0$, since it is a continuous function $X$. By equation (\ref{C1}), there exists a sequence of real numbers $\ep(t_i)>0$ converging to zero as $t_i\rightarrow 0$, such that 
$$ - \ep(t_i) \tilde{\omega}_{\beta} \leq dd^c \left\{\frac{1}{t_i}(\varphi_{t_i} - \varphi) - \varphi'  \right\} \leq \ep(t_i) \tilde{\omega}_{\beta}, $$
and the LHS of equation (\ref{IBP-time}) can be estimated as 
\begin{eqnarray}
\label{IBP-time3}
\left| \int_{X-D} \varphi dd^c \left\{ \frac{1}{t}(\varphi_t - \varphi) - \varphi' \right\}      \wedge\omega_{\varphi}^{n-j-1}\wedge\omega^j \right|
&\leq& \ep(t) \int_{X-D} |\varphi | \tilde{\omega}_{\beta}\wedge\omega_{\varphi}^{n-j-1}\wedge\omega^j
\nonumber\\
&\leq& \ep(t) C,
\end{eqnarray}
for some uniform constant $C$. Finally we can do integration by parts as 
\begin{equation}
\label{IBP-time1}
\int_{X-D} (\varphi_t - \varphi) dd^c\varphi\wedge\omega_{\varphi}^{n-j-1}\wedge\omega^j 
= \int_{X-D} \varphi dd^c(\varphi_t - \varphi)\wedge\omega_{\varphi}^{n-j-1}\wedge\omega^j,
\end{equation}
by theorem 1.14 of \cite{BEGZ10} again, and the claim follows.
\end{proof}

\begin{rem}
\label{rem-energy}
If $\alpha$ is a fixed conic closed $(1,1)$ current in the definition of $\mathcal{E}^{\alpha}$, then lemma (\ref{lem-1variation}) still holds. This is because we can always write $\alpha$ as the difference of two conic K\"ahler metrics(maybe in different cohomology class) with $\mathcal{C}^{2,\alpha,\beta}$ potentials.  
\end{rem}
Next, let's consider how to define the conic entropy functional, and here we will use a slightly different formula: for a conic K\"ahler potential $\varphi\in\mathcal{C}^{2,\alpha,\beta}$, put
\begin{equation}
\label{cone-entropy}
H_{c}(\varphi): = \int_{X} \log \frac{\omega_{\varphi}^n}{\omega^n/ |s|^{2-2\beta}_{\psi}} \omega_{\varphi}^n,
\end{equation}
where $\psi$ is a smooth $psh$ metric on the line bundle $L_D$. Note that the function $\frac{\omega_{\varphi}^n |s|_{\psi}^{2-2\beta}}{\omega^n}$ is uniformly bounded on $X$ by the local isometric property of conic metrics. Therefore, the functional $H_c$ is well defined. 
When the curve $\omega_{\varphi_t}$ is in the family $\mathcal{R}$, we can compute its first variation:
\begin{eqnarray}
\label{entropy4}
\frac{d}{dt}H_c(\varphi_t) &=& \int_{X-D} \Delta_{\varphi}\varphi'\omega_{\varphi}^n +
\int_{X-D} \log \frac{\omega_{\varphi}^n}{e^{-(1-\beta)\log |s|^2_{\psi}}\omega^n} \Delta_{\varphi}\varphi' \omega_{\varphi}^n
\nonumber\\
&=& n\int_{X-D} dd^c\varphi'\wedge\omega_{\varphi}^{n-1} + n\int_{X-D} \log\frac{\omega_{\varphi}^n}{\omega^n / |s|_{\psi}^{2-2\beta}} dd^c\varphi'\wedge\omega_{\varphi}^{n-1} 
\nonumber\\
&=& n\int_{X-D} \varphi' dd^c \log\frac{\omega_{\varphi}^n |s|_{\psi}^{2-2\beta}}{\omega^n} \wedge\omega_{\varphi}^{n-1}
\nonumber\\
&=& n \int_{X-D} \varphi' (Ric\omega - Ric\omega_{\varphi} + (1-\beta)[D] - dd^c\psi)\wedge\omega_{\varphi}^{n-1}
\nonumber\\
&=& n\int_{X-D} \varphi' (Ric\omega - Ric\omega_{\varphi} )\wedge\omega_{\varphi}^{n-1}  - n\int_{X-D} \varphi' dd^c\psi\wedge \omega_{\varphi}^{n-1}.
\end{eqnarray}

Here we need to justify the equality from the second line to the third line in equation (\ref{entropy4}). The integration by parts formula can be applied to 
the first term $\int_{X-D} dd^c\varphi'\wedge\omega_{\varphi}^n$ for the same reason as in equation (\ref{IBP-time}). For the second term of line two, the bounded function in the integrant can be written as
$$\log\frac{\omega_{\varphi}^n |s|_{\psi}^{2-2\beta}}{\omega^n} = (1-\beta)\log|s|^2_{\psi} - \phi_{ric},$$
where $(1-\beta)\log|s|^2_{\psi}$ is a $quasi$-$psh$ function on $X$, and the Ricci potential can be decomposed into $\phi_{ric} = \Psi - \psi $ as in lemma (\ref{lem-ric}). 
Then the term in line three follows from theorem 1.14 of \cite{BEGZ10}.

Finally we can define the conic Mabuchi functional as follows:
\begin{defn}
\label{def-cMabuchi}
The conic Mabuchi functional $\mathcal{M}_c$ acting on the space of all $\mathcal{C}^{2,\alpha,\beta}$ conic K\"ahler potentials is defined as 
\begin{equation}
\label{cMabuchi}
\mathcal{M}_c: = c(\beta)\mathcal{E} - \mathcal{E}^{Ric\omega - dd^c\psi}  + H_c, 
\end{equation}
where $\psi$ is the smooth metric on the line bundle $L_D$.
\end{defn}

\begin{corollary}
\label{cor-mabuchi}
Suppose $\varphi_t$ is a $\mathcal{C}^1$ curve in the space of $\mathcal{C}^{2,\alpha,\beta}$ conic K\"ahler potentials, 
and assume $\omega_{\varphi_t} \in\mathcal{R}$. Then the first variation of the conic Mabuchi functional can be computed as follows:
\begin{equation}
\label{cMabuchi2}
\frac{d}{dt}\mathcal{M}_c(\varphi_t) = n \int_{X-D} \varphi'_t(c(\beta)\omega_{\varphi} - Ric\omega_{\varphi})\wedge\omega_{\varphi}^{n-1}.
\end{equation}
Hence conic cscK metrics are critical points of $\mathcal{M}_c$. 
\end{corollary}
\begin{proof}
Combine lemma (\ref{lem-1variation}) and equation (\ref{entropy4}), and the result follows.
\end{proof}

Notice that we can extend the domain of our conic entropy functional to the space $\mathcal{H}_{\Delta, \beta}$. Consider the function 
$x = \log\frac{\omega_{\varphi}^n |s|_{\psi}^{2-2\beta}}{\omega^n} $. Then $x\log x$ is bounded from above and below on $X$ for any $\varphi\in \mathcal{H}_{\Delta,\beta}$. And the conic entropy can be written as 
$$H_c (\varphi) = \int_{X} (x\log x) d\mu_0 , $$ where the measure $d\mu_0 = \omega^n/|s|_{\psi}^{2-2\beta}$ has finite mass on $X$. Hence it is well defined. 
Finally, we can also extend the domain of conic Mabuchi's functional to this larger space, since the energy part is always well defined once the potential is 
a bounded $quasi$-$psh$ function thanks to the Bedford-Talyor product. Therefore, we infer the following.

\begin{corollary}
For any $\varphi\in \mathcal{H}_{\Delta,\beta}$, the conic Mabuchi functional $\mathcal{M}_c(\varphi)$ is well defined.
In particular, we can restrict $\mathcal{M}_c$ to any $\cC^{1,\bar 1}_\b$ geodesic $\cG$. 
\end{corollary}

\subsection{Complex Hessian of conic energies}
In order to prove the convexity(subharmonicity) of the conic Mabuchi functional, we will investigate its second order derivative in this section. First, we claim that the energy part still behaves very well on the conic geodesics.

\begin{lemma}
\label{lem-second}
Let $\varphi_{\tau} = \Phi(\tau, \cdot)$ be the K\"ahler potentials of a geodesic connecting two conic K\"ahler metric. 
Then the complex Hessian of the energy functionals can be computed as
\begin{equation}
\label{hessian2}
d_{\tau}d^c_{\tau} \mathcal{E}(\varphi_{\tau}) = \int_{X} (\pi^*\omega + dd^c\Phi)^{n+1}; \ \ \ 
d_{\tau}d^c_{\tau} \mathcal{E}^{\alpha}( \varphi_{\tau} )= \int_{X} (\pi^*\omega + dd^c\Phi)^{n}\wedge\pi^*\alpha.
\end{equation}
\end{lemma} 
\begin{proof}
Notice that $\Phi|_{X\times\{ \tau\}} \in \mathcal{H}_{\Delta, \beta}$, for each $\tau\in \Sigma$, and then $\mathcal{E}(\varphi_{\tau})$ and 
$\mathcal{E}^{\alpha}(\varphi_{\tau})$ are well defined as discussed before. 
Since $\varphi$ are H\"older continuous functions on $X\times \Sigma$, there exists a smooth approximation sequence 
$\varphi_{\ep}$ decreasing to $\varphi$ uniformly on any open domain of $X\times\Sigma$ \cite{BK07}. 

Let $\chi(\tau)$ be a smooth compactly supported testing function on $\Sigma$, and by equation (\ref{hessian}) we have 
$$ \int_{\Sigma} (d_{\tau}d^c_{\tau} \chi) \mathcal{E}(\varphi_{\ep}) = \int_{X\times\Sigma} \chi (\pi^*\omega + dd^c\Phi_{\ep})^{n+1}.$$
Thanks to the Bedford-Talyor approximation theorem, we can pass to limits on the RHS side
$$ \int_{X\times\Sigma} \chi (\pi^*\omega + dd^c\Phi_{\ep})^{n+1} \rightarrow \int_{X\times\Sigma} \chi (\pi^*\omega + dd^c\Phi)^{n+1}.$$
Moreover, we have
$$\int_{\Sigma} (d_{\tau}d^c_{\tau}\chi )\mathcal{E}(\varphi_{\ep}) = \int_{\Sigma} \mathcal{E}(\varphi_{\ep}) dd^c\chi
\rightarrow \int_{\Sigma}  \mathcal{E}(\varphi) dd^c\chi , $$
by dominate convergence theorem, and our result follows.
\end{proof}

Now on the conic geodesic $\mathcal{G}$, let's try to take the complex Hessian of the conic Mabuchi functional in the current sense as follows. 
\begin{eqnarray}
\label{a}
\int_{\Sigma}\chi d_{\tau}d^c_{\tau} \mathcal{M}_c(\varphi_{\tau}) &=& -\int_{X\times\Sigma} \chi (Ric\omega - dd_{X}^c\psi)\wedge\mathcal{G}^n
\nonumber\\
& + & \int_{X\times\Sigma} \chi dd_{t,X}^c  \log \frac{\omega_{\varphi}^n}{\omega^n/ |s|^{2-2\beta}_{\psi}} \wedge\mathcal{G}^n,
\end{eqnarray}
where $\omega_{\varphi} = \mathcal{G}|_{X\times\{t\}}$, and the operator $d_{t,X}$ means taking derivatives on both fiber and time directions.
In order to prove convexity, it is enough to show the following inequality among currents:
\begin{equation}
\label{convex1}
dd_{t,X}^c \log \frac{\omega_{\varphi}^n}{\omega^n/ |s|^{2-2\beta}_{\psi}} \wedge\mathcal{G}^n \geq (Ric \omega - dd_{X}^c\psi)\wedge\mathcal{G}^n.
\end{equation}

\section{ Convexity of conic Mabuchi's functional }
\label{sec-4}
In this section, we will generalize our previous work \cite{CLP} to the conic setting. The basic idea is to construct a new approximation of the volume element of the geodesic $\mathcal{G}$, by means of solving a family of Monge-Amp\`ere equations. 
And readers are encouraged to read our previous work, because the idea used here is very similar to \cite{CLP}. 

Let's explain some notations for the use of later calculation. 
\begin{enumerate}

\item[(i)] we will abuse the notation $\omega_{\beta}$ for local model conic metric and the global one $\tilde{\omega}_{\beta}$(equation (\ref{533})). So we have 
$\omega_{\beta} = \omega + dd^c\varphi_{\beta}$. \\

\item[(ii)] since all energies only depend on the real part of the complex variable $\tau\in \Sigma$, it is reasonable to abuse $\tau$ and its real part $t$.\\

\item[(iii)] we write $d_X$ or $\partial_{X}$ as for taking derivatives on the fiber $X\times\{ t\}$ for some $t$, and $d_{t,X}$ or $\partial_{t,X}$ as for taking derivatives on
the product manifold $X\times \Sigma$. If there is no confusion, then we just use $d$ or $\partial$.\\

\item[(iv)] the constant $C$, $C'$... might be different from line to line, and we hope its self-explanation is clear enough.
\end{enumerate}

\subsection{Regularization}
Let $\omega_{\varphi_i}, i=1,2$ be two conic K\"ahler metrics, with $\varphi_i \in \mathcal{C}^{2,\alpha,\beta}$. 
Suppose there exists a conic $\mathcal{C}_{\beta}^{1,\bar{1}}$ geodesic $\mathcal{G}$ connecting these two points.
First, it can be approximated by taking the convolution with respect to mollifiers. Take $\Sigma'$ to be a relatively compact open subset of the annuals $\Sigma$. 
\begin{lemma}
\label{lem-app-geod}
On a slight smaller open domain $X\times \Sigma'$ in the product manifold $X\times\Sigma$, there exists a sequence $\mathcal{G}_{\delta}$ to approximate the geodesic $\mathcal{G}$ satisfying:
\begin{enumerate}

\item[(i)] the $(1,1)$ form $\mathcal{G}_{\delta}$ are non-singular, and there is a uniform constant $C$, such that we have:
$$ \mathcal{G}_{\delta}  \geq -2C\delta (\omega + \sqrt{-1} dt\wedge d\bar{t}) , $$
on $X\times\Sigma'$.\\


\item[(ii)] the coefficients ($g_{\delta, \alpha\bar{\beta}}, g_{\delta, \alpha\bar{t}}, g_{\delta, t\bar{\beta}}, g_{\delta, t\bar{t}}$) of  $\mathcal{G}_{\delta}$ are uniformly bounded in the sense of conic K\"ahler metrics. That is, in a local coordinate $(z_1,\cdots, z_n)$ near a point $p$ on the divisor, we have 
\begin{eqnarray}
\label{conic C11}
\sum_{\alpha,\beta>1}|g_{\delta, \alpha\bar{\beta}}|+ |g_{\delta, t\bar{t}}| &+& 2|z_1|^{1-\beta} (|g_{\delta, 1\bar{t}}|+ \sum_{\beta>1}|g_{\delta, 1\bar{\beta}}|)
\nonumber\\
&+& |z_1|^{2-2\beta}|g_{\delta, 1\bar{1}}| + 2\sum_{\beta>1}|g_{\delta, t\bar{\beta}}| < C,
\end{eqnarray}
where the uniform constant $C$ doesn't depend the point $p$, or the index $\delta$. \\

\item[(iii)] the coefficients satisfy the following convergence property on each fiber $X\times \{t\}$: \\
$g_{\delta,\alpha\bar{\beta}}\rightarrow g_{\alpha\bar{\beta}} $ for $\alpha,\beta >1$, in 
$L^{p}$ for any $p>1$; \\
$g_{\delta,1\bar{1}}\rightarrow g_{1\bar{1}}$ in $L^p$ for  $1< p < \frac{1}{1-\beta}$; \\
$g_{\delta, 1\bar{\beta}}\rightarrow g_{1\bar{\beta}}$ for $\beta >1$, and $g_{\delta, 1\bar{t}}\rightarrow g_{1\bar{t}}$ in $L^{p}$ for $1<p<\frac{2}{1-\beta}$;\\
$g_{\delta, t\bar{t}}\rightarrow g_{t\bar{t}}$ and $g_{\delta, t\bar{\beta}} \rightarrow g_{t\bar{\beta}} $ for $\beta>1$ in $L^p$ for any $p>1$. \\
And all the convergences are uniform with respect to $t\in \Sigma'$.\\

\item[(iv)] for potentials, we have 
$$\sup_{t\in\Sigma'} ||\varphi_{\delta} - \varphi ||_{\mathcal{C}^0(X\times\{ t\})} \rightarrow 0,$$
as $\delta\rightarrow 0$.

\end{enumerate}
\end{lemma}
\begin{proof}
The geodesic can be written as $\mathcal{G} = \omega + dd^c_{t,X}\varphi$, for a conic $\mathcal{C}_{\beta}^{1,\bar{1}}$ potential $\varphi$. Then we can construct a smooth approximation sequence $\mathcal{G}_{\delta}$ as in theorem 2 of \cite{BK07}. On each coordinate chart $U_{\alpha}$, we can take convolution of $\varphi$ with respect to a local smoothing kernel $\rho_{\delta}$ \cite{BK07}. And then glue each piece together to get $\varphi_{\delta}$, by using regularized-maximum function. Finally define 
$$\mathcal{G}_{\delta}: = \omega + dd^c_{t,X}\varphi_{\delta} .$$ 

From the basic properties of convolution and the fact $0\leq\mathcal{G}\leq\mathcal{G}_{\beta}$, we have 
$$ \mathcal{G}_{\delta}\geq -C\delta ( \omega + \sqrt{-1} dt\wedge d\bar{t}), $$


And (ii) follows from the property of convolution on $psh$ functions ( see Appendix, Proposition (\ref{prop-glue})) and then (iii) and (iv) follow from (ii).

\end{proof}

In fact, we can take a closer look at the convolution process, and have a better growth control for the coefficients of the smoothing geodesics as follows.

\begin{lemma}
\label{lem-grow}
In a fixed small neighborhood $U$ of the divisor $D$,
the coefficients of $\mathcal{G}_{\delta}|_{X\times\{t\}}$ grows as:
\begin{enumerate}

\item[(a)]
for points $x\in U$ such that $\emph{dist}(x,D) < 3\delta$, we have
$$ 0 \leq g_{\delta,1\bar{1}} \leq C \delta^{2\beta-2}, \ \ \ |g_{\delta,1\bar{\beta}}| \leq C \delta^{\beta-1}, $$
where $C$ is a uniform constant independent of $\delta$ and $t\in \Sigma'$.\\

\item[(b)] for points $x\in U$ such that $\emph{dist}(x, D) \geq 3\delta$, we have 
$$ 0 \leq g_{\delta, 1\bar{1}} \leq \frac{C}{ ( |s|^2 + \delta^2 )^{1-\beta}}; \ \ \ |g_{\delta, 1\bar{\beta}}|^2\leq  \frac{C}{ ( |s|^2 + \delta^2 )^{1-\beta}},$$
for some uniform constant $C$.
\end{enumerate}
\end{lemma}
\begin{proof}
As before, let $p$ be a point on the divisor $D$, and $B_R$ is a coordinate ball centered at $p$ with radius $R$, such that locally $D$ is given by $\{z_1=0\}$. 
And we will prove the growths control of the coefficients of $\mathcal{G}_{\delta}$ uniformly in this ball $B_R$.

First notice that it's enough to prove $0\leq g_{\delta,1\bar{1}}\leq C\delta^{2\beta-2}$, and then the growth of $|g_{\delta,1\bar{\beta}}|$ is given 
by the positivity of $\mathcal{G}_{\delta}$ as 
$$ g_{\delta,1\bar{1}}g_{\delta, \beta\bar{\beta}} \geq |g_{\delta, 1\bar{\beta}}|^2. $$
Now only normal direction to the divisor is under consideration, and we can restrict our attention to this one dimensional disk by Fubini's theorem. 
Writing $ u = g_{1\bar{1}} \geq 0$, we have by convolution
$$ g_{\delta, 1\bar{1}} = u_{\delta} = \int u(z - \delta w) \rho(w) dV(w), $$
for $\delta < R/10$, and $\rho(w)$ is standard mollifier with support on the unit ball. 
Then there are three cases to discuss:
\begin{enumerate}

\item[(1)] $z=0$;

\item[(2)] $ |z| \geq 3\delta $;

\item[(3)] $0 < |z| < 3\delta $.

\end{enumerate}
For $(1)$, we can compute as 
\begin{eqnarray}
\label{growth}
u_{\delta}(0) &=& \int u(-\delta w)\rho(w)dw\wedge d\bar{w} 
\nonumber\\
&\leq& \frac{C}{\delta^{2-2\beta}} \int_{|w|<1} \frac{1}{|w|^{2-2\beta}} dw\wedge d\bar{w}
\nonumber\\
&\leq & C \delta^{2\beta-2}.
\end{eqnarray}
For $(3)$, notice that the ball $B_{\delta}(z)$ is contained in $B_{4\delta}(0)$ for all $|z|< 3\delta$. After enlarging the integration domain,
we have $ u_{\delta} \leq C (4\delta)^{2\beta-2}$ by the same computation as in equation (\ref{growth}).

For $(2)$, notice that $|z-\delta w|^2 > |z/3|^2 + (\delta/3)^2$ for all $|z|\geq 3\delta$. Therefore, we have 
\begin{eqnarray}
\label{grow2}
\int u(z-\delta w)\rho(w)dw\wedge d\bar{w} &\leq& \frac{3^{2\beta-2}C}{(|z_1|^2 + \delta^2)^{1-\beta}} \int \rho(w)dw\wedge d\bar{w}
\nonumber\\
&\leq & \frac{C'}{(|s|^2 + \delta^2)^{1-\beta}}.
\end{eqnarray}

Next we claim that the first derivative in the normal direction of $\varphi_{\delta}$ is controlled as 
$ \left|  \partial \varphi_{\delta} / \partial z_1\right|^2(x) \leq C \delta^{2\beta-2},$ 
for $\text{dist}(x,D) < 3\delta$; and
$$ \left|  \partial \varphi_{\delta} / \partial z_1\right|^2(x) \leq \frac{C}{(|s|^2 + \delta^2)^{1-\beta}}, $$
for $\text{dist}(x, D) \geq 3\delta$. 
This is because the geodesic potential $\varphi$ belongs to the H\"older space $C^{1,\alpha}$ in $w$-coordinates, due to the $\mathcal{C}^{1,\bar{1}}_{\beta}$ regularity.
Hence we have the growth control $|\partial \varphi / \partial z_1|^2 < C |z_1|^{2\beta -2}$ near the divisor, 
and the claim follows from a similar argument as above.

Finally, we need to take regularized maximum for different coordinate charts. 
For simplicity, let's consider a point $x$ contained only in two charts $U_{\alpha}$ and $U_{\beta}$.
Assume $\varphi_{\alpha}$ and $\varphi_{\beta}$ are local trivializations of $\varphi$ on $U_{\alpha}$ and $U_{\beta}$ near $x$. 
Take convolutions as $\varphi_{\delta,\alpha} = \varphi_{\alpha}*\rho_{\delta}$ and $\varphi_{\delta,\beta} = \varphi_{\beta}*\rho_{\delta}$. 
Put $M_X$ to be the regularized maximum operator for two variables. 
Notice that the value, the first and second order derivatives of $M_{X}$ only depend on the geometry of $X$ (see Proposition (\ref{prop-b}), Appendix). 
Therefore, we have the same growth rates on its derivatives with the previous estimates of local convolutions

\end{proof}

In order to find a better approximation,
we will introduce a triple of index $\eta: = (\ep, \delta, \ep')$, for small positive real numbers $\ep$, $\delta$ and $\ep'$, and in prior, they are independent of each other. Define a semi-positive $(1,1)$ form as 
$$\mathcal{G}_{\delta}' := \mathcal{G}_{\delta} + 2C\delta(\omega+ \sqrt{-1}dt\wedge d\bar{t}) > 0.$$
Put $\Delta_{\eta}: = \mathcal{G}_{\delta}' / \ep$, and then the class $c_1(X) + \{ \Delta_{\eta} \} |_{X\times\{ t\}}$ is K\"ahler on each fiber. 
Then we can solve the following family of Monge-Amp\`ere equations by the famous result of Yau \cite{Yau} as 
\begin{equation}
\label{MA}
(\Theta_{\omega}(K_X + D) + \Delta_{\eta} + dd^c \phi_{t,\eta})^n = \ep^{-n} e^{\phi_{t,\eta}} \frac{\omega^n} {( |s|^2_{\psi}+ \ep' )^{1-\beta} }, \ \ \ \ (\dagger_{\eta})
\end{equation}
on each fiber $X\times\{t \}$. And the solution $\phi_{t,\eta}$ satisfies the positivity condition:
$$ \rho: = \left.(  \Theta_{\omega}(K_X+D) + \Delta_{\eta} + dd^c \phi_{t,\eta}  ) \right|_{X\times{\{t\}}} > 0. $$ 

According to theorem 4.1 of \cite{CLP}, we infer that 
\begin{equation}
\label{positive}
 \Xi: =  \Theta_{\omega}(K_X+D) + \Delta_{\eta} + dd^c \phi_{t,\eta} > 0, 
\end{equation}
as a smooth closed positive $(1,1)$ form on the product manifold $X\times\Sigma'$. 
In fact, it is enough to prove the Ricci curvature of the K\"ahler form $\rho$ is bounded below by $-1$. 
But this is true thanks to the positivity of the following form
\begin{equation}
dd^c \log (|s|^2 + \ep'e^{\psi}) \geq 0, 
\end{equation}
for any positively curved smooth metric $\psi$ on the line bundle $D$.

\subsection{ $C^0$ estimate and equicontinuity } 
We will prove the $C^0$ estimate for the family of equations $(\dagger_{\eta})$. And the claim is as follows.
\begin{lemma}
\label{lem-C0}
Let $\phi_{t,\eta}$ be the solution of equation $(\dagger_{\eta})$. Then there exits a proper choice of $\ep' = \ep'( \delta)$, 
where $\ep'\rightarrow 0$ if $\delta\rightarrow 0$, such that on each fiber $X\times\{t \}$ we have:
$$ \sup_{X}\phi_{t,\eta} \leq C,\ \ \  -\ep \inf_{X} \phi_{t,\eta} \leq C, $$ for a uniform constant $C$ independent of $\eta$ and $t\in \Sigma'$.
\end{lemma}
\begin{proof}
We are going to apply maximum principle on the fiber $X\times\{ t\}$. Rewriting the Monge-Amp\`ere equation as 
\begin{equation}
\label{MA1}
(\ep\Theta + \mathcal{G}'_{\delta} + \ep dd^c\phi_{t,\eta} )^n = \frac{e^{\phi_{t,\eta}} \omega^n}{ (|s|^2_{\psi} + \ep')^{1-\beta}}
\end{equation}

Let the point $p\in X$ be as $\phi_{t,\eta}(p) = \max_{X} \phi_{t,\eta}$, and then $dd^c\phi_{t,\eta}(p) \leq 0$ by maximum principle. This implies 
$( \ep\Theta + \mathcal{G}'_{\delta} )(p) > 0$, and we have 
$$ \phi_{t,\eta} \leq \phi_{t,\eta}(p) \leq \log \frac{(\ep\Theta + \mathcal{G}'_{\delta})^n}{d\mu_{\beta,\ep'}}(p), $$
where $ d\mu_{\beta,\ep'} = \omega^n / (|s|^2_{\psi} + \ep')^{1-\beta}$. Now consider the following function 
$$ f(\delta,\ep'): = \frac{( I + \mathcal{G}'_{\delta})^n (|s|_{\psi}^{2} + \ep' )^{1-\beta}}{\omega^n} (p),$$
where $I$ is the identity matrix, and we can always assume $\ep\Theta \leq I$. 

The sequence of volume forms $( I+\mathcal{G}'_{\delta})^n$ 
is a smooth approximation of the singular volume form $(I + \mathcal{G})^n$, which is uniformly bounded by a constant $C'$ times $\omega^n/ |s|^{2-2\beta}_{\psi}$. 
Thanks to lemma (\ref{lem-grow}), we can find $\ep'(\delta)$ as function of $\delta$ decreasing fast enough such that 

$$ (I + \mathcal{G}'_{\delta})^n \leq \frac{C' k(\delta)}{(|s|^2_{\psi} + \ep'(\delta))^{1-\beta}} \omega^n,$$
where $\ep'(\delta)\rightarrow 0$ and $k(\delta)\rightarrow 1$ as $\delta\rightarrow 0$. 
In fact, put $\ep' = \delta^2$, and then point (a) of lemma (\ref{lem-grow}) implies that 
$$ (I + \mathcal{G}'_{\delta})^n (|s|^2 + \delta^2)^{1-\beta} /\omega^n \leq 10^{2-2\beta}C, $$  
for points $x\in U$ such that $\text{dist}(x, D) < 3\delta$. And also point (b) of lemma (\ref{lem-grow}) directly implies 
$$ (I + \mathcal{G}'_{\delta})^n (|s|^2 + \delta^2)^{1-\beta} /\omega^n \leq C, $$
for points $x\in U$, such that $\text{dist}(x, D)\geq 3\delta$. 
Hence we conclude that $f(\delta,\ep'(\delta))\leq 2C'$ for all $\delta$ small enough, and the upper bound follows since $\phi_{t,\eta} \leq f(\delta)$. 

In order to find the lower bound, we need to rewrite equation (\ref{MA}) again. First let $\varphi_{\beta,\ep'}$ be the solution of the following Monge-Amp\`ere equation
\begin{equation}
\label{MA2}
(\omega + dd^c\varphi_{\beta,\ep'})^n = d\mu_{\beta,\ep'}.
\end{equation}
The solution exists because of Yau's result \cite{Yau} again, and by Kolodziej \cite{Kol}, the H\"older norm $|\varphi_{\beta, \ep'}|_{\mathcal{C}^{\alpha}} $ is uniformly bounded independent of $\ep'$. Now put
$$ \Omega_{\beta, \ep'}: = \omega + 2C\delta (\omega + \sqrt{-1}dt\wedge d\bar{t}) + \frac{1}{2}dd^c\varphi_{\beta,\ep'},$$
and $$\tau_{\eta}: = \ep\phi_{t,\eta} + \varphi_{\delta} - \frac{1}{2}\varphi_{\beta,\ep'}.$$
Then equation (\ref{MA}) can be rewritten as 
\begin{equation}
\label{MA4}
(\ep\Theta + \Omega_{\beta,\ep'} + dd^c \tau_{\eta})^n = \exp\{ \frac{1}{\ep} (\tau_{\eta} - \varphi_{\delta} + \frac{1}{2}\varphi_{\beta,\ep'}  ) \} d\mu_{\beta, \ep'}.
\end{equation}
By the minimum principle, on the point $q\in X$ where $\tau_{\eta}(q) = \min_{X} \tau$, we have
$$ \tau_{\eta} \geq \tau_{\eta}(q)\geq \ep\log\frac{(\ep \Theta + \Omega_{\beta,\ep'})^n}{d\mu_{\beta,\ep'}} (q) + \varphi_{\delta}(q) - \frac{1}{2}\varphi_{\beta,\ep'}(q), $$
where $\ep$ is small enough such that $\ep\Theta + \omega/2 >0$. Therefore, we have
 $$(\ep\Theta + \Omega_{\beta,\ep'})^n \geq 2^{-n}(\omega_{\beta,\ep'})^n \geq 2^{-n} d\mu_{\beta,\ep'},$$
and the lower bound of $\ep\phi_{t,\eta}$ follows since $\varphi_{\delta}$ and $\varphi_{\beta,\ep'}$ are all uniformly bounded.
\end{proof}

\begin{lemma}
\label{lem-equi}
For each $\ep$, the family $\phi_{t,\eta}$ obtained by piecing together fiber-wise solutions is equicontinuous.
\end{lemma}
\begin{proof}
The proof is very similar to what we did before in theorem 4.2 of \cite{CLP}, and we recall it here for the convenience of readers. For fixed $\ep$, put 
$\Psi_{t,\eta}: = \ep \Xi |_{X\times\{t\}}$, and then we can compare equation (\ref{MA}) at $t_0$ and $t_1$ for the same index $\delta$ and $\ep'$:
$$\left( \Psi_{t_0, \eta} + dd^c(\xi_{\eta}(t_1) - \xi_{\eta}(t_0)) \right)^n  =
 \exp\{ \frac{1}{\ep}(\xi_{\eta}(t_1) - \xi_{\eta}(t_0) -\varphi_{\delta}(t_1) + \varphi_{\delta}(t_0)  )\} (\Psi_{t,\eta})^n, $$
where $\xi_{\eta}(t) = \ep\phi_{t,\eta} + \varphi_{\delta}$. By maximum principle and property (iv) of lemma (\ref{lem-app-geod}), 
we claim that the difference of $\ep\phi_{t,\eta}$ can be estimated as 
$$ \sup_X |\phi_{t_0,\eta}(x) - \phi_{t_1, \eta}(x) | \leq \ep^{-1}C | t_0 - t_1|. $$
This is because for fixed $x$, we have 
$  \left| \frac{\partial}{\partial t} \varphi (t,x) \right| < C$
for some uniform constant $C$ independent of $x$ and $t$. Therefore, we infer the following estimate: 
$$ | \varphi_{\delta}(t_0, x) - \varphi_{\delta}(t_1, x)  | \leq C |t_0 - t_1|, $$ for some uniform constant on $X\times\Sigma'$, and proved the claim. 

Then apply maximum principle on the fiber $X\times\{ t \}$ in a similar manner, and we have  
$$ |\phi_{t_0,\eta}(x_0) - \phi_{t_1,\eta}(x_1)| \leq \ep^{-1} C (|t_0 - t_1| +  d(x_0, x_1)^{\alpha'}), $$
where $C$ is a uniform constant independent of $\eta$, and $d(x_0, x_1)$ is the distance between the two points measured by the model cone metric.
\end{proof}

\subsection{$C^2$ estimate}
Guenancia and P\u aun generalized a prior $C^2$ estimate for K\"ahler-Einstein equation to conic case in their work \cite{GP14}. Here we
will invoke this method again.
\begin{prop}[Guenancia, P\u aun]
\label{pr-C2}
Let $\omega_{\varphi} = \omega + dd^c\varphi$ be a K\"ahler metric satisfying 
\begin{equation}\label{544}
(\omega + dd^c\varphi)^n = e^{\psi_+ - \psi_{-}}\omega^n,
\end{equation}
for some smooth function $\psi_{\pm}$. Suppose there exists a uniform constant $C>0$, and a smooth function $\Psi$ such that 
\begin{equation}
\label{5455}
C\omega+dd^c\Psi \geq 0,
\end{equation}
and
\begin{equation}
\label{545}
i\Theta_{\omega}(T_X) \geq - (C\omega + dd^c\Psi)\otimes \emph{Id}.
\end{equation}
Then we have 
\begin{equation}
\label{lap}
\Delta_{\omega_{\varphi}} (\log tr_{\omega}\omega_{\varphi} + \Psi - A\varphi) \geq 
\frac{ \Delta_{\omega}f }{tr_{\omega}\omega_{\varphi} }  - C tr_{\omega_{\varphi}}\omega -nA + A( tr_{\omega_{\varphi}}\omega),
\end{equation}
where $f = \psi_+ - \psi_{-}$.
\end{prop}

Although equation (\ref{lap}) is in a global form, it is in fact computed in a local normal coordinate near a point $p\in X$. So proposition (\ref{pr-C2}) applies even if the Monge-Amp\`ere equation (\ref{544}) holds locally. 
In order to apply this inequality, we prefer to write our family of Monge-Amp\'ere equations as 
$$(\omega_{\ep'} + dd^c\varphi_{\ep'})^n = e^{\mu \varphi_{\ep'} + F_{\ep'}} \omega_{\ep'}^n,  $$
with 
$$ F_{\ep'}: = -\log\frac{(|s|^2 + \ep')^{1-\beta} \omega^n_{\ep'}}{dV}. $$
Here $\omega_{\ep'} = \omega + dd^c \varphi_{\ep'}$ is constructed in the paper of Guenancia and P\u aun \cite{GP14}.
It is equivalent to our $\omega_{\beta,\ep'}$, 
and will approximate the model conic K\"ahler metric $\omega_{\beta}$ as $\ep'\rightarrow 0$. 
However, we have some extra term in the Monge-Amp\`ere operator in equation $(\dagger_{\eta})$ like $\ep\Theta_{\omega}(K_X)$ and so on. 
Hence our goal is prove these perturbations won't really affect the estimate. First, put 
$$ \tau'_{\eta}: = \ep\phi_{\eta} + \varphi_{\delta}  - \varphi_{\ep'}  $$
and define locally
$$ \zeta_{\eta}: = \tau'_{\eta} + \ep\psi_{\theta}, $$
where $\psi_{\theta}$ is the local potential of the $(1,1)$ form $\Theta_{\omega}(K_X)$. Then we can
re-write equation (\ref{MA}) locally as 
\begin{equation}
\label{MA6}
(\omega_{\ep'} + dd^c \zeta_{\eta})^n =  
\exp \left( \frac{1}{\ep} (\zeta_{\eta} - \varphi_{\delta}  + \varphi_{\ep'} - \ep\psi_{\theta} ) + F_{\ep'}\right) \omega_{\ep'}^n.
\end{equation}
Here we forgot about the positive $(1,1)$ form $2C\delta(\omega + \sqrt{-1}dt\wedge d\bar{t})|_{X\times\{ t\}}$, since the metric 
$\omega_{\ep'} + 2C\delta(\omega + \sqrt{-1}dt\wedge d\bar{t})$ is a small smooth perturbation(only depends on the index $\delta$!) of the non-degenerate metric $\omega_{\ep'}$. In fact, the growth of the curvature of the perturbed metric is completely determined by $\omega_{\ep'}$ itself near the divisor.

Take $\omega_{\zeta_{\eta}} = \omega_{\ep'} + dd^c \zeta_{\eta}$ and 
$f = \frac{1}{\ep}u_{\eta} + F_{\ep'}$, where 
$$u_{\eta}: = \zeta_{\eta} - \varphi_{\delta}  + \varphi_{\ep'} = \ep (\phi_{\eta} + \psi_{\theta}).  $$

Although $u_{\eta}$ is not globally defined, $\Delta_{\omega}u_{\eta}$ is globally well define with respect to any K\"ahler form $\omega$.
Hence $\Delta_{\omega}f$ is also globally defined. 
Applying proposition (\ref{pr-C2}) to equation (\ref{MA6}) under the curvature assumption (\ref{5455}) and (\ref{545}), we get
\begin{eqnarray}
\label{566}
\Delta_{\omega_{\zeta}} (\log tr_{\omega_{\ep'}}\omega_{\zeta}  + \Psi_{\ep'}  -A\tau'_{\eta}) &\geq& 
A\ep \Delta_{\omega_{\zeta}} \psi_{\theta} -  \frac{\Delta_{\omega_{\ep'}} \psi_{\theta}}{tr_{\omega_{\ep'}} \omega_{\zeta}}      + \frac{\Delta_{\omega_{\ep'}} f}{tr_{\omega_{\ep'}} \omega_{\zeta}}
\nonumber\\
& +& (A - C) tr_{\omega_{\zeta}}\omega_{\ep'} - nA
\nonumber\\
&\geq & (\ep A -1)\frac{\Delta_{\omega_{\ep'}} \psi_{\theta}}{tr_{\omega_{\ep'}} \omega_{\zeta}} + \frac{\Delta_{\omega_{\ep'}} f}{tr_{\omega_{\ep'}} \omega_{\zeta}}
+(A - C) tr_{\omega_{\zeta}}\omega_{\ep'} - nA
\nonumber\\
&\geq& \frac{\Delta_{\omega_{\ep'}} f}{tr_{\omega_{\ep'}} \omega_{\zeta}} + (A -C -C'|\ep A -1|) tr_{\omega_{\zeta}}\omega_{\ep'} - nA,
\end{eqnarray}
for $\ep A < 1$, and $C$ and $C'$ are some uniform constants. For the second inequality in equation (\ref{566}), we assumed $\Theta = dd^c\psi_{\theta} \geq 0$,
since on the point where $dd^c\psi_{\theta} <0$, we have an even better lower bound.

Now the two sides of inequality (\ref{566}) become globally defined, and we can apply maximum principle to it as in \cite{GP14}.

\begin{lemma}
\label{lem-C2}
Let $\phi_{t,\eta}$ be the solution of equation $(\dagger_{\eta})$. Then there exits a proper choice of $\ep' = \ep'( \delta)$, 
where $\ep'\rightarrow 0$ if $\delta\rightarrow 0$, such that we have:
$$ -C\omega_{\ep'}\leq \ep dd^c\phi_{\ep,\delta,\ep'} \leq C\omega_{\ep'}$$ for a uniform constant $C$ independent of $\eta$.
\end{lemma}
\begin{proof}
According to section 4 of \cite{GP14}, a very useful auxiliary function $\Psi$ is constructed as 
$$\Psi_{\ep', \rho}: = C \chi_{\rho} (|s|^2_{\psi} + \ep'), $$ where $\chi_{\rho}$ is a kind of cut-off function defined on $[\ep' , +\infty [$ to $\mathbb{R}$ such that
$$i\ddbar \Psi_{\ep', \rho}  \geq C (|s|^2_{\psi} + \ep')^{\rho-1} \sqrt{-1} \langle \partial^{\psi}s, \partial^{\psi}s \rangle - C\omega_{\ep'}. $$
In particular, this implies 
$$ C\omega_{\ep'} + dd^c\Psi_{\ep'} \geq 0. $$

As also proved in \cite{GP14}, for any fixed $0<\rho<\min \{ \beta, 1-\beta \}$, the function $\Psi_{\ep'} = \Psi_{\ep',\rho}$ is uniformly bounded,
and we have the following two important inequalities:
\begin{equation}
\label{577}
i\Theta_{\omega_{\ep'}}(T_X) \geq - (C\omega_{\ep'} + dd^c\Psi_{\ep'})\otimes \text{Id};
\end{equation}
and
\begin{equation}
\label{588}
dd^c F_{\ep'} \geq - (C\omega_{\ep'} + dd^c \Psi_{\ep'})
\end{equation}
where $C$ is a constant independent of $\ep'$. 

Now equation (\ref{577}) allows us to use proposition (\ref{pr-C2}), and we infer from equation (\ref{588}) as 
$$ \Delta_{\omega_{\ep'}} F_{\ep'} \geq -nC - \Delta_{\omega_{\ep'}}\Psi_{\ep'},$$
and then by equation (\ref{566}) we have 
$$ \Delta_{\omega_{\zeta}}\Psi_{\ep'} \geq  - \frac{\Delta_{\omega_{\ep'}}F_{\ep'}}{tr_{\omega_{\ep'}}\omega_{\zeta}}  - C{tr_{\omega_{\zeta}}\omega_{\ep'}} .$$
Hence we can re-write equation (\ref{5455}) as 
\begin{eqnarray}
\label{599}
\Delta_{\omega_{\zeta}} (\log tr_{\omega_{\ep'}}\omega_{\zeta}  + 2\Psi_{\ep'}  -A\tau'_{\eta}) 
&\geq& \frac{\Delta_{\omega_{\ep'}} ( \frac{1}{\ep}u_{\eta} ) }{tr_{\omega_{\ep'}} \omega_{\zeta}} + (A -2C -C'|\ep A -1|) tr_{\omega_{\zeta}}\omega_{\ep'} - nA
\nonumber\\
&\geq& \frac{\Delta_{\omega_{\ep'}} ( \frac{1}{\ep}u_{\eta} ) }{tr_{\omega_{\ep'}} \omega_{\zeta}} + (A - 2C - C')tr_{\omega_{\zeta}}\omega_{\ep'} - nA
\nonumber\\
&\geq& \frac{\Delta_{\omega_{\ep'}} ( \frac{1}{\ep}u_{\eta} ) }{tr_{\omega_{\ep'}} \omega_{\zeta}} - nA.
\end{eqnarray}
The second line works for $\ep$ small enough such that $\ep (2C+C'+1) <1$, under the assumption $A < 2C + C'+1$. And in the third line, we just take $A = 2C+ C'$.
Observe that $dd^c u_{\eta} = \omega_{\zeta} - \mathcal{G}'_{\delta}$, 
and then at the maximum point $p$, we have 
$$ (1-\ep nA) ( tr_{\omega_{\ep'}}\omega_{\zeta}) (p) \leq tr_{\omega_{\ep'}} \mathcal{G}'_{\delta}(p)  $$
Finally, as we did in the proof of the upper bound of $C^0$ estimate, take $\ep'(\delta)$ decreasing fast enough such that 
$$ \mathcal{G}'_{\delta} \leq C\omega_{\ep'(\delta)}, $$ 
for some uniform constant $C$, when $\delta\rightarrow 0$. This is possible thanks to lemma (\ref{lem-grow}). 
Then for $\ep$ small enough, we have 
$$  tr_{\omega_{\ep'}}\omega_{\zeta} (p) \leq \frac{C}{1-\ep nA},$$
and the $C^2$ estimate follows.
\end{proof}

\subsection{Approximation of the conic geodesic}
Based on our previous $C^0$ and $C^2$ estimate (lemma (\ref{lem-C0}) and (\ref{lem-C2})), we conclude as follows.
\begin{lemma}
\label{cor-geod}
For each fixed $\ep>0$, we can extract a limit 
$$\lim_{\delta\rightarrow 0} \phi_{\ep, \delta, \ep'(\delta)}  = \phi_{\ep},$$
strongly in $C^0$. And the restriction of $\phi_{\ep}$ to the fiber $X\times\{ t\}$ is the unique conic $\mathcal{C}^{1,\bar{1}}$ solution of the degenerate Monge-Amp\`ere
equation
\begin{equation}
\label{epsilon}
(\Theta_{\omega}(K_X + D)  + \ep^{-1}\mathcal{G} + dd^c\phi_{t,\ep})^n = \ep^{-n} e^{\phi_{t,\ep}}\frac{\omega^n}{|s|^{2-2\beta}_{\psi}}.
\end{equation}
Moreover we have 
$$ \Theta_{\omega}(K_X + D) + \ep^{-1}\mathcal{G} + dd^c\phi_{\ep} \geq 0, $$
as a current in the product manifold $X\times\Sigma'$.
\end{lemma}
\begin{proof}
The last statement is because of the equi-continuity proved in lemma (\ref{lem-equi}). In fact, the limit 
$\phi_{\ep}: = \lim_{\delta\rightarrow 0} \phi_{\ep, \delta, \ep'(\delta)}$ is continuous in the interior of $X\times\Sigma$, 
and the global positivity from equation (\ref{positive}) passes down to 
$$  \Theta_{\omega}(K_X + D) + \ep^{-1}\mathcal{G} + dd^c\phi_{\ep} \geq 0.$$
\end{proof}
Next, we need a lemma to investigate the behavior of the solution $\phi_{\ep}$ when $\ep\rightarrow 0$. 
The method is used in theorem (4.3) of \cite{CLP} before, and we recall it here.
\begin{lemma}
\label{lem-final}
The sequence of measures 
$$ e^{\phi_{t,\ep}}\omega^n / |s|^{2-2\beta}_{\psi}  $$ converges to $\mathcal{G}^n|_{X\times\{t \}}$ weakly in $L^p$, for any $1< p < \frac{1}{1-\beta} $.
\end{lemma}
\begin{proof}
First we re-write equation (\ref{epsilon}) as 
\begin{equation}
\label{epsilon1}
(\omega + \ep \Theta_{\omega}(K_X) + dd^c\tau_{\ep} )^n = e^{\frac{1}{\ep} (\tau_{\ep} - \varphi)}\frac{\omega^n}{|s|^{2-2\beta}_{\psi}},
\end{equation}
where $\tau_{\ep}: = \ep\phi_{t,\ep} + \varphi$. Now thanks to the $C^0$ and $C^2$ estimates, there exists some conic $\mathcal{C}^{1,\bar{1}}$ function $\rho\leq 0$ such that 
$$\ep\phi_{\ep}\rightarrow \rho, $$ in $\mathcal{C}^{1,\alpha,\beta}$, when $\ep\rightarrow 0$. Hence the limit $\tau = \rho + \varphi$ is smaller than $\varphi$.

In the same time, we have the following convergence for Monge-Amp\`ere masses thanks to (iii) of lemma (\ref{lem-app-geod}): 
$$ (\omega + \ep\Theta + dd^c\tau_{\ep})^n \rightarrow (\omega + dd^c\tau)^n, $$
in sense of weak $L^p$ for any $1< p < 1/(1-\beta)$. 

Put $U_{\delta}: = \{ \tau < \varphi - \delta \}$ (an open set of $X$). By comparison principle \cite{K08}, \cite{GZ07}, we have 
$$\int_{U_{\delta}} \mathcal{G}^n \leq \int_{U_{\delta}} (\omega + dd^c\tau)^n. $$
However, on the set $U_{\delta}$, the inequality $\tau_{\ep} - \varphi < -\delta/2$ holds for $\ep$ small enough. Hence by equation (\ref{epsilon1}), we infer
$$ \int_{U_{\delta}} (\omega + dd^c\tau)^n = 0 .$$

Next, we claim that $\mathcal{G}^n|_{\ol{U}} =0$, where $\ol{U}$ is the closure of the open set $\{ \tau < \varphi\}$. This is because the set $\ol{U} - U$ has measure zero, and the density $\omega^n / |s|^{2-2\beta}_{\psi}$ is in $L^1$. Hence we have $\mathcal{G}^n = (\omega + dd^c\tau)$ as currents on $X$.
Finally we invoke the uniqueness result in \cite{K08} to prove $\rho=0$, and our lemma follows.
\end{proof}

\subsection{Interior convexity}
In order to enhance the weak convergence to strong convergence, we invoke the Banach-Saks theorem here as in \cite{CLP}. Choose $t_i \in (0,1)$ with $i=0,1,2$ be three arbitrary points, and then we fix them. By lemma (\ref{lem-final}) and the Banach-Saks, one can find a sequence $\ep_{k}$ as $k\rightarrow +\infty$ such that 
$$ \frac{1}{k}\sum_{j=1}^{k} \exp (\phi_{t_i,\ep_j})\frac{\omega^n }{ |s|^{2-2\beta}_{\psi}} \rightarrow \mathcal{G}^n|_{X\times\{ t_i\}},$$
strongly in $L^p$ with some $p>1$, for $i=0,1,2$. In prior, the sequence $\ep_k$ does depend on $t_i$, but it becomes harmless by the following lemma.
\begin{lemma}
\label{lem-533}
Let $f_i\in L^p$ converges weakly to a function $h$ in $L^p$, and $g_i\in L^p$ converges weakly to a function $k$ in $L^p$. Then
there is a convex combination $\mathcal{F}$, consisting of composition with picking up subsequence and taking average, such that 
$$\mathcal{F}(\{ f_i\}) \rightarrow h,\ \ \ \ \mathcal{F}(\{ g_i\}) \rightarrow k,$$
simultaneously in strong $L^p$ sense. 
\end{lemma}
\begin{proof}
For the sequence $\{f_1, f_2,\cdots  \}$, we can pick up a subsequence $\{f_j\}$ such that the first convex combination 
$$ \mathcal{F}_1(\{ f_i \}) = \frac{1}{k} \sum_{j=1}^k f_j \rightarrow h$$ strongly thanks to Banach-Saks. If apply this convex combination to another sequence, then it still
converges weakly as 
$$\mathcal{F}_1(\{ g_i \}) \dashrightarrow k.$$
Now use Banach-Saks again on the new sequence $\mathcal{F}_1(\{g_i \})$, and we have another convex combination $\mathcal{F}_2$ such that 
$$ \mathcal{F}_2\circ \mathcal{F}_1(\{ g_i \}) \rightarrow k,$$ strongly in $L^p$. Notice that $\mathcal{F}_2$ acting on a strongly converging sequence is still 
strongly converging. Therefore, put $\mathcal{F}: = \mathcal{F}_2\circ \mathcal{F}_1$, and we have $\mathcal{F}(\{f_i \})\rightarrow h$ strongly in the same time.
\end{proof}

By repeating lemma (\ref{lem-533}), we can find a convex combination $\mathcal{F}_{\phi_{\ep}}$ such that 
$$\mathcal{F}_{0,1,2}(\{ \exp (\phi_{t_i, \ep})\omega^n / |s|^{2-2\beta}_{\psi} \}) \rightarrow \mathcal{G}^n|_{X\times\{t_i \}}$$
strongly in $L^p$, for $i=0,1,2$ simultaneously. Hence we can define the following volume element as 
$$ \omega^n_{t,k}:= \mathcal{F}^{(k)}_{0,1,2}(\{ \exp (\phi_{t, \ep})\omega^n / |s|^{2-2\beta}_{\psi} \}) $$
where $\mathcal{F}^{(k)}$ means the $k$th term in the convex combination as a sequence. 

Let $\mathcal{M}_k$ be the conic Mabuchi functional evaluated on the conic geodesic $\mathcal{G}$, with the entropy modified by 
the volume element $\omega_{t,k}^n$ as 
\begin{equation}
\label{Mk}
\mathcal{M}_k(t): = c(\beta)\mathcal{E}(\varphi_t) - \mathcal{E}^{Ric\omega - dd^c\psi}(\varphi_t) 
+ \int_{X}\log \frac{\omega^n_{t,k}}{\omega^n/|s|^{2-2\beta}_{\psi}}\mathcal{G}^n
\end{equation}

The new sequence $\mathcal{M}_k$ is an approximation to $\mathcal{M}_c$ at least at these three points $t_0$, $t_1$ and $t_2$. And
we claim that it is a convex function along the conic geodesic.
\begin{lemma}
\label{lem-Mk}
For each $k\geq 1$, the functional $\mathcal{M}_k$ is a continuous convex function on $[0,1]$.
\end{lemma}
\begin{proof}
The functional $\mathcal{M}_k(t)$ is continuous, thanks to the regularity we got from $C^0$ and $C^2$ estimates for each $\phi_{t,\ep}$.
Then it is enough to show that $\mathcal{M}_k$ is convex in the weak sense. 
And this boils down to the inequality 
$$ dd^c \log\frac{\omega^n_{t,k}}{\omega^n/|s|^{2-2\beta}_{\psi}} \wedge \mathcal{G}^n \geq ( Ric\omega - dd^c\psi )\wedge \mathcal{G}^n. $$ 
For each $k$, writing the convex combination as 
$$\omega^n_{t,k} = \sum_{j=1}^k \lambda_{j}^k \exp(\phi_{t,\ep_j})\omega^n / |s|^{2-2\beta}_{\psi},$$
where $\sum_{j=1}^k \lambda_j^k =1$. Then we have 
$$ dd^c \log \left( \sum_{j=1}^k \lambda_j^k \exp(\phi_{\ep_j}) \right) \geq \sum_{j=1}^k 
\frac{\lambda_j^k \exp(\phi_{\ep_j})}{\sum_{l=1}^k \lambda_{l}^k\exp(\phi_{\ep_l})} dd^c\phi_{\ep_j},  $$
by direct computation. Based on the positivity in lemma (\ref{cor-geod}), we have 
$$ -Ric(\omega) + dd^c\psi+ \ep_j^{-1}\mathcal{G} + dd^c\phi_{\ep_j} \geq 0.$$
Finally, we obtain 
$$dd^c \log \left( \sum_{j=1}^k \lambda_j^k \exp(\phi_{\ep_j}) \right)\wedge\mathcal{G}^n
\geq  \sum_{j=1}^k \frac{\lambda_j^k \exp(\phi_{\ep_j})}{\sum_{l=1}^k \lambda_{l}^k\exp(\phi_{\ep_l})} (Ric\omega - dd^c\psi) \wedge\mathcal{G}^n,  $$
which proves the lemma.
\end{proof}

\begin{corollary}
\label{cor-convex}
The conic Mabuchi functional $\mathcal{M}_c(t)$ on the conic geodesic is a convex function on the interval $(0,1)$. 
\end{corollary}

\subsection{Continuity at the boundary}
The entropy function $H_c(t)$ evaluated on the conic geodesic is lower semi-continuous up to the boundary.
In fact, we invoke the following lemma from our previous paper \cite{CLP}.

\begin{lemma}[Chen-Li-P\u aun] 
\label{Chen}
Suppose $f, f_i$ are uniformly bounded non-negative functions, such that $f_i\rightarrow f$ weakly in $L^1$,
and assume $\mu$ is a positive locally finite measure. Then  
\[
\lim_{i} \int_X ( f_i\log f_i - f\log f) d\mu \geqslant 0.
\]
\end{lemma}

Base on this semi-continuity lemma, we conclude our theorem as follows.

\begin{theorem}
\label{555}
Suppose two conic K\"ahler potential $\varphi_0$ and $\varphi_1$ can be connected by a conic $\mathcal{C}^{1,\bar{1}}_{\beta}$ geodesic $\mathcal{G}$. Then
the conic Mabuchi functional $\mathcal{M}_c(t)$ restricted to $\mathcal{G}$ is a convex continuous function on the closed interval $[0,1]$.
\end{theorem}

Invoking the existence theorem in \cite{CZ12}, we immediately get the following convexity theorem. 

\begin{corollary}
\label{cor-convex}
If two conic K\"ahler potentials $\varphi_0, \varphi_1$ are in the space  $\mathcal{H}_C$, then the conic Mabuchi functional $\mathcal{M}_c(t)$ restricted to $\mathcal{G}$ is a convex continuous function on the closed interval $[0,1]$.
\end{corollary}

\section{Appendix }
\label{sec-app}
\subsection{Blocki-Kolodziej regularization   }

Let $\varphi$ be a $\omega$-$psh$ functions on a complex K\"ahler manifold $X$, and $\omega$ is its K\"ahler form. 
Suppose the Lelong number $\nu(\varphi)$ of $\varphi$ is everywhere zero on $X$. Then Blocki-Kolodziej \cite{BK07} constructed 
a sequence of smooth $(1+\ep_j)\omega$-$psh$ functions $\varphi_j$ on an open subset $M\subset X$,
decreasing to $\varphi$ at every point when $\ep_j\searrow 0$. 
And we will recall the last step of their proof here. 

Take $\rho(z)$ to be the standard mollifier on $\mathbb{C}^n$, and put $\rho_{\delta} = \delta^{-2n}\rho(z/\delta)$ for $\delta>0$. 
Then we denote by $\varphi_{\delta}$ the convolution of $u$ by $\rho_{\delta}$
$$ u_{\delta} (z): = u\ast \rho_{\delta} (z) = \int_{\mathbb{C}^n} u(z - \delta w) \rho(w) d\lambda(w).$$

Let $\{V_\a \}, \a \in \mathcal{A}$ be a finite number of coordinate charts covering $M$, and $V_\a \subset U_{\a}$. 
Suppose $f_\a$ is the local potential of the K\"ahler form $\omega$ on the chart $U_{\a}$. Then $u_\a: = \varphi + f_\a$ has 
the local regularization $u_{\a,\delta}$ in the chart $U_\a$. Then they can prove 
\begin{equation}
\label{app-001}
u_{\a,\delta} - u_{\b,\delta} \rightarrow f_\a - f_\b 
\end{equation}
locally uniformly in $U_\a \cap U_{\b}$. Let $\eta_\a + 1$ be a smooth compactly supported function in $U_\a$ such that $\eta_\a = 0$ on $V_\a$. 
Then there exists a uniform constant $C$ such that $dd^c \eta_\a > -C\omega$, and we can take 
$$ \varphi_{\t} := \max_\a \{ \phi_{\a,\t} \},  $$
where
$$ \phi_{\a,\t}: = u_{\a,\t} - f_\a + \ep \eta_\a /C.  $$
Notice that the constant $C$ does NOT depend on $\ep$ in our case, since the background metric $\omega$ is a smooth K\"ahler form here.

Moreover, near the boundary of $U_{\a}$, the local function $\phi_{\a,\t}$ does not contribute to the maximum since $\eta_{\a} = -1$ there.
Hence we can take the so called regularized Maximum operator $M_{\tau}$ instead of taking maximum.
Here $\tau: = (\tau_\a)$ is a collection of real positive numbers $\{\tau_{\a}\}$ (depending on $\ep$) for each $\a\in \mathcal{A}$.
Therefore, we have a family of smooth $(1+\ep)\omega$-$psh$ 
functions 
$$ \Phi_{\tau, \t}: = M_{\tau} \{ \phi_{\a,\t}\} $$
such that $\Phi_{\tau,\t}$ decreases to $\varphi$ by a diagonal argument of $\delta$ and $\ep$. 

\subsection{Derivatives of regularized maximum}
Recall the definition of the regularized maximum operator first. 
Let $\theta\in C^{\infty}(\mathbb{R},\mathbb{R})$ be a non-negative function supported on $(-1,1)$, such that $\int_{\mathbb{R}} \theta(h)dh =1$ 
and $\int_{\mathbb{R}} h\theta(h) dh =0$. For arbitrary $\tau = (\tau_1,\cdots, \tau_p) \in (0, +\infty)^p$, we denote by $M_{\tau}$ the regularized maximum operator 
\begin{equation}
\label{rmax}
M_\tau(t_1,\cdots, t_p): = \frac{1}{\prod_{1\leq j\leq p} \tau_j}\int_{\mathbb{R}^p} \max \{t_1 + h_1,\cdots, t_p+ h_p \} \prod_{1\leq j\leq p} \theta \left(\frac{h_j}{\tau_j}\right) 
dh_1\cdots dh_p. 
\end{equation}

Note that $M_{\tau}$ is non-decreasing for each variable $t_j, 1\leq j\leq p$,
and it is a smooth and convex function on $\mathbb{R}^p$ \cite{abook}. 
 
Next, we want to investigate the behaviors of derivatives of this operator. 
For simplicity, we take $p=2$, and $(x,y) \in \mathbb{R}^2$ is the Euclidean coordinate. 
Now make the following change of variables
\begin{equation}
\label{2222}
z: = (x+y)/2,\ \ \ w:= (x-y)/2,
\end{equation}

\begin{lemma}
\label{lem-max}
Assume $\tau = \tau_1 = \tau_2$, and then we have 
\begin{equation}
\label{007}
|\d_{z}M_\tau| +  |\d_w M_{\tau}| + |\d_z\d_z M_{\tau}| + |\d_{z}\d_w M_\tau| < C,
\end{equation}
Moreover, 
\begin{equation}
\label{008}
|\d_w\d_w M_{\tau}| \leq C\tau^{-1}, 
\end{equation}
for some uniform constant $C$, not depending on $\tau$ and $(x,y)\in \mathbb{R}^2$.
\end{lemma}
\begin{proof}
There exists another smooth approximation of maximum operator 
$$ \max\{ x, y\} = \lim_{r\rightarrow +\infty} f_r (x, y),$$
where 
$$f_r(x,y): = \frac{1}{r}\log (e^{rx} + e^{ry}). $$
Take derivatives inside of the integral ($\max (x,y)$ is continuous) in equation (\ref{rmax}), and we have 
$$ \d_{z}f_{r} = 1,\ \ \ \d_w f_r = \frac{e^{rx} - e^{ry}}{ e^{rx} + e^{ry}},\ \ \  \d_z\d_z f_r =0,\ \ \ \d_w\d_z f_r =0. $$ 
Then equation (\ref{007}) follows by taking $r\rightarrow +\infty$. On the other hand, observe that 
it is enough to estimate $\d_w\d_w M_\tau$ on the diagonal $\{(x,x)\in \mathbb{R}^2 \}$ since $M_{\tau}$ is a convex function. 
But we have
$$ \d_w\d_w \max\{z+w, z-w\} = \delta_{\{w=0 \}}$$ 
in the distributional sense, where $\delta_{\{ w=0 \}}$ represents the delta function at $\{w=0 \}$.

Now if we put
$$ \theta(x/\tau)\theta(y/\tau) = \rho(z,w,\tau), $$
then $\rho(z,w,\tau)$ is a smooth bounded function supported in the unit rectangle. 
\begin{eqnarray}
\label{009}
\d_w\d_w M_{\tau} (z, 0) &=& \frac{1}{\tau^2} \int_{-\infty}^{+\infty}dz \int_{-\infty}^{+\infty} \delta_{\{w=0\}} \rho(z,w,\tau) dw
\nonumber\\
&=&  \frac{1}{\tau^2} \int_{-\infty}^{+\infty} \rho(z,0,\tau) dz
\nonumber\\
&\leq& C \tau^{-1}.
\end{eqnarray}
The last line follows from the fact that the function
$$\rho(z,0,\tau) = \theta(x/\tau)^2 = \tilde{\rho}(z/\tau)$$
is another mollifier on $z$ coordinate. 
\end{proof}

\subsection{Regularization of $\cC^{1,\bar 1} (\cC^{1,\bar 1}_\b)$  metrics }
First let $\cG$ denote by a $\cC^{1,\bar 1}$ metric on $X$. That is to say, $\cG$ is a positive closed $(1,1)$ current such that
$$ 0\leq \cG \leq C\omega. $$
Writing $\cG: = \omega + dd^c \varphi$, the $\omega$-$psh$ function $\varphi$ has bounded Laplacian on $X$. 
Hence the $C^{1,\a}$ H\"older norm of $\varphi$ is uniformly bounded.

Now suppose $V_{\a} \subset U_\a$ forms a finite cover of an open subset $M\subset X$, and locally we have 
$$ \cG|_{U_\a} = dd^c u_\a,$$
for a family of $psh$ function $u_\a$ on each $U_\a$. 
Put $u_{\a,\delta}$ as the local regularization of $u_\a$ by the standard mollifier on $\mathbb{C}^n$. 
For each small $\ep >0$, we also take 
$$\phi_{\a,\t} := u_{\a,\t} - f_\a + \ep \eta_\a /C, $$
as we did in Blocki-Kolodziej regularization.

Now the observation is that for each $\ep$ fixed, the gap between the boundary value $u_{\a}(z), z\in \d U_\a$
and interior values $u_{\b}(z)$ for all $\b$ such that $z\in V_\b$ is larger than $\ep/2$ for all $\delta$ small enough. 
That is to say, for each point $z\in\d \Omega_\a$ and all $ \a\in\mathcal{A}$, we can choose $\tau = \ep/4$ such that the following inequality holds
\begin{equation}
\label{glue}
 u_\a(z) + \ep/4 < \max_{z\in V_\b} \{ u_{\b} (z) - \ep/4\} \leq \max_{z\in U_\b} \{ u_\b(z) - \ep/4 \},
\end{equation}
for all $\delta$ small enough.
Then thanks to the glueing principle (Corollary 5.19, \cite{abook} ),
the regularized maximum 
$$\Phi_{\ep/4,\t} = M_{\ep/4}\{\phi_{\a,\t}\}$$ 
is well defined for all $\delta$ small enough. 

\begin{prop}
\label{prop-glue}
For each $\ep>0$, there exists a small $\delta >0$, such that we have 
$$ | \d\ol\d \Phi_{\ep/4, \delta} | \leq C,$$
for a uniform constant $C$. Moreover, we have 
$$ \d\ol\d \Phi_{\ep/4, \delta} - \d\ol\d \varphi  \rightarrow 0,$$
in $L^p$ sense, for any $p>1$. 
\end{prop}
\begin{proof}
For simplicity, we assume that there are only two coordinate charts $V_{\a}\subset U_\a$ and $V_{\b}\subset U_\b$ intersecting with each other near a point $p$, 
and the general case follows in a similar way. 
Then the regularized maximum operator is a function on $(x,y)\in \mathbb{R}^2$, and we introduce the change of variables $(z,w)\in\mathbb{R}^2$ by equation (\ref{2222}). 
That is to say, if take 
$$z = (\phi_{\a,\t} + \phi_{\b,\t})/2,\ \ \  w = (\phi_{\a,\t} - \phi_{\b,\t})/2, $$
then thanks to Lemma (\ref{lem-max}), we have
\begin{eqnarray}
\label{11111}
\d_k\d_{\bar l} M_{\ep/4} \{ z(\phi_{\a,\t}, \phi_{\b,\t}), w(\phi_{\a,\t}, \phi_{\b,\t})\} &=& \d_k\d_{\bar l}z + \d_w M_{\ep/4}\cdot \d_k\d_{\bar l}w
\nonumber\\
& + & \d_w\d_w M_{\ep/4} \cdot \d_k w \d_{\bar l} w.
\end{eqnarray}
It is easy to see that $\d_w M_{\ep/4}$ converges to zero when $w\rightarrow 0$.  
But for the last term, we need a more carful investigation. 
Take $$v_{\a, \b, \t}: = (u_{\a,\t} - u_{\b,\t}) - (f_\a - f_\b), $$
and 
$$\eta_{\a, \b}: = \eta_{\a} - \eta_\b, $$ on $U_\a\cap U_\b$. 
Note that we have 
$$ \d v_{\a, \b, \t}  \rightarrow 0, $$ in some $C^{\gamma}$ H\"older space when $\delta$ converges to zero. 
On the other hand, we have  
\begin{eqnarray}
\label{33333}
\ep^{-1} \d_k w \d_{\bar l} w &=&
\ep^{-1} \d (v_{\a, \b, \t} + \ep \eta_{\a,\b}) \cdot \ol\d  (v_{\a, \b, \t} + \ep \eta_{\a, \b} )
\nonumber\\
& = & \ep^{-1} \d v_{\a,\b,\t} \ol\d v_{\a,\b,\t} + \d v_{\a,\b,\t} \ol\d\eta_{\a,\b} 
\nonumber\\
&+& \d \eta_{\a,\b}\ol\d v_{\a,\b,\t} + \ep \d \eta_{\a,\b}\ol\d \eta_{\a,\b}. 
\end{eqnarray}
Now for each $\ep > 0$, we pick up a $\delta$ (depending on $\ep$), such that 
$$   |\d v_{\a,\b,\t}|^2 \leq C \ep^2, $$ for some uniform constant $C$. 
Thanks to Lemma (\ref{lem-max}),
for such a pair $(\ep,\t)$, there exists a function $\delta: = \delta(\ep)$ such that  
we have 
$$ |\d_w\d_w M_{\ep/4}\cdot \d_k w \d_{\bar l} w| \rightarrow 0, $$ as $\ep$ converges to zero, and the results follow. 

\end{proof}

Now let $\cG$ be a $\cC^{1,\bar 1}_\b$ metric. That is to say, $\cG$ has bounded Laplacian with respect to a K\"ahler metric $\omega_\b$ on $X$
with cone singularity of angle $2\pi\b$ along a smooth divisor $D$.  
It is enough to estimate the approximation near the divisor, since everything is the same as in Proposition (\ref{prop-glue}) outside a tubular neighborhood of $D$. 
Note that in this tubular neighborhood, we can indeed arrange the finite covering in such a way that 
any point on the divisor belongs to at almost two open coordinate charts $V_\mu$ and $V_\g$.

Take arbitrary point $p\in D$, and an open neighborhood $U$ with coordinate $\{\z, z_2,\cdots, z_n\}$ such that $D: = \{\z =0\}$ in $U$. 
Applying the Blocki-Kolodziej regularization to $\cG$, we obtain the following estimates in the normal direction near the divisor (tangential directions and mixed directions can be estimated in an easier way). 
 
\begin{prop}
\label{prop-b}
For each $\ep>0$, there exists a small $\delta>0$, such that we have 
$$ |\d_\z\d_{\bar\z} \Phi_{\ep/4,\t}| \leq C \t^{2\b-2} $$
\end{prop}
\begin{proof}
The only issue here is to understand the convergence of $\d v_{\a,\b,\t}$ in conic case.
But this is determined by the difference between the convolution $\varphi_{\a,\t}$ and $\varphi|_{U_\a}$.
Thanks to Fubini's theorem, we can restrict our calculation in the normal direction, 
Note that $\d \varphi \in \cC^{,\a,\b}$. That is to say, if we write $\d \varphi(\z) = \tilde{f} (|\z|^{\b-1}\z)$, then $\tilde{f}$ is $C^{\a}$ H\"older continuous.
Therefore, we have 
\begin{eqnarray}
\label{55555}
\d( \varphi_{\a,\t} - \varphi) (p) &=&   \{ (\d\varphi )_{\t}- \d\varphi \} (p)
\nonumber\\
&=&  \int_{\mathbb{C}}  (\d \varphi(\t \z) - \d \varphi(0)) \rho(\z) d\lambda(\z)
\nonumber\\
&\leq & C\t^{\a\b + \b -1} \int_{\mathbb{C}} |\z|^{\a\b + \b -1} \rho (\z) d\lambda(\z)
\nonumber\\
&\leq& C'\t^{\a\b + \b -1}.
\end{eqnarray}
Moreover, for points $z$ in the small tubular neighborhood of $D$, we have the same the growth estimate with equation (\ref{55555})
by a similar argument.  
Hence we have 
$$ |\t|^{2-2\b}|\d_w\d_w M_{\ep/4} \cdot \d_\z w\d_{\bar\z} w | \rightarrow 0,  $$
when $\ep$ converges to zero, and the result follows.
\end{proof}

\end{document}